\documentclass[12pt]{amsart}
\usepackage{graphicx} 

\usepackage{soul} 
\usepackage{amscd,amssymb,mathtools,aliascnt}
\usepackage[dvipsnames]{xcolor}
\usepackage{graphicx}
\usepackage{calrsfs}
\usepackage[labelfont=rm]{subcaption}
\usepackage{tikz-cd} 
\usepackage{tikz}
\usetikzlibrary{decorations.markings}
\usetikzlibrary{decorations.pathreplacing}
\usetikzlibrary{shapes}
\usetikzlibrary{backgrounds}
\usetikzlibrary{calc}
\usetikzlibrary{shapes.misc, positioning}
\usepackage{mdwlist}
\usepackage{xfrac}
\DeclareCaptionSubType*{figure}

\usepackage[arrow,matrix,graph,frame,poly,arc,tips]{xy}

\usepackage{multirow}
\usepackage{enumerate}
\usepackage{verbatim}
\usepackage{comment}
\usepackage{todonotes}

\DeclarePairedDelimiter{\form}{\langle}{\rangle}

    \newcommand\ba{\begin{align*}}
    \newcommand\ea{\end{align*}}
    \newcommand\be{\begin{enumerate}}
    \newcommand\ee{\end{enumerate}}
    \newcommand\bpf{\begin{proof}}
    \newcommand\epf{\end{proof}}
    \newcommand\bpp{\begin{prop}}
    \newcommand\epp{\end{prop}}
    \newcommand\bpb{\begin{prob}}
    \newcommand\epb{\end{prob}}
    \newcommand\bd{\begin{defn}}
    \newcommand\ed{\end{defn}}
    \newcommand\bh{\begin{hint}}
    \newcommand\eh{\end{hint}}

    \newcommand\N{\mathbb{N}}
    \newcommand\bR{\mathbb{R}}
    \newcommand\bQ{\mathbb{Q}}
    \newcommand\bZ{\mathbb{Z}}
    \newcommand\Z{\mathbb{Z}}

    \newcommand\CC{\mathcal{C}}
    
    \newcommand\FF{\mathcal{F}}

    \newcommand\JJ{\mathcal{J}}

    \newcommand{\R}[0]{\mathbb{R}}

\DeclareMathOperator\PL{PL}
\DeclareMathOperator\Aff{Aff}
\newcommand\pls{\operatorname{PL}_+(S^1)}
\DeclareMathOperator\hol{hol}
\DeclareMathOperator\BP{BP}

\newcommand\Isom{\operatorname{Isom}}
    
    \newcommand\Comm{\operatorname{Comm}}
    
    \newcommand\rot{\operatorname{rot}}
    \newcommand\Rot{\operatorname{Rot}}
    
    \newcommand\supp{\operatorname{supp}}

    \newcommand\gam{\Gamma}

    \DeclareMathOperator\Homeo{Homeo}

    \newcommand\sse{\subseteq}
    \newcommand\co{\colon}

    \DeclareMathOperator\Fix{Fix}

    \DeclareMathOperator\PSL{PSL}
    
    \DeclareMathOperator\SO{SO}
    \DeclareMathOperator\Diff{Diff}

    \DeclareMathOperator\stab{stab}

    \DeclareMathOperator\inte{int}
    \DeclareMathOperator\cl{cl}

    \def\thetitle{Hyperbolicity and obstructions to PL actions of the circle}
    \def\theauthors{{Leonardo Dinamarca, Maximiliano Escayola, Sang-hyun Kim, Thomas Koberda}}
    \usepackage{hyperref}
    \hypersetup{
      colorlinks=false,
      plainpages,
      urlcolor=black,
      linkcolor=black,
      pdftitle=   \thetitle,
      pdfauthor=  {\theauthors}
    }
    \usepackage[capitalise]{cleveref}

    \theoremstyle{plain}
    
    \newtheorem{thm}{Theorem}[section]
    \crefname{thm}{Theorem}{Theorems}
    \newaliascnt{lem}{thm}
    \newtheorem{lem}[lem]{Lemma}
    \aliascntresetthe{lem}
    \crefname{lem}{Lemma}{Lemmas}
    \newaliascnt{cor}{thm}
    \newtheorem{cor}[cor]{Corollary}
    \aliascntresetthe{cor}
    \crefname{cor}{Corollary}{Corollaries}
    \newaliascnt{prop}{thm}
    \newtheorem{prop}[prop]{Proposition}
    \aliascntresetthe{prop}
    \crefname{prop}{Proposition}{Propositions}
    \newaliascnt{con}{thm}
    
    \aliascntresetthe{con}
    \crefname{con}{Conjecture}{Conjectures}
    \newaliascnt{que}{thm}
    
    \aliascntresetthe{que}
    \crefname{que}{Question}{Questions}
    \newaliascnt{obs}{thm}
    
    \aliascntresetthe{obs}
    \crefname{obs}{Observation}{Observations}
    \newaliascnt{hyp}{thm}
    \newtheorem{hyp}[hyp]{Hypothesis}
    \aliascntresetthe{hyp}
    \crefname{hyp}{Hypothesis}{Hypotheses}
    \newtheorem*{claim*}{Claim}

    \theoremstyle{remark}
    \newaliascnt{exmp}{thm}
    
    \aliascntresetthe{exmp}
    \crefname{exmp}{Example}{Examples}
    \newaliascnt{rem}{thm}
    \newtheorem{rem}[rem]{Remark}
    \aliascntresetthe{rem}
    \crefname{rem}{Remark}{Remarks}

    \theoremstyle{definition}
    \newaliascnt{defn}{thm}
    \newtheorem{defn}[defn]{Definition}
    \aliascntresetthe{defn}
    \crefname{defn}{Definition}{Definitions}
    \newtheorem{prob}{Problem}[section]
    \newaliascnt{setting}{thm}
    
    \aliascntresetthe{setting}
    \crefname{setting}{Setting}{Settings}

\begin{document}
    \title\thetitle
    \date{\today}

    
    \author[L. Dinamarca]{Leonardo Dinamarca}
    \address{School of Mathematics, Korea Institute for Advanced Study (KIAS), Seoul, 02455, Korea}
    \email{}
    \urladdr{}
    
    \author[M. Escayola]{Maximiliano Escayola}
    \address{School of Mathematics, Korea Institute for Advanced Study (KIAS), Seoul, 02455, Korea}
    \email{maxiescayola@kias.re.kr}    
    \urladdr{}
    
    \author[S. Kim]{Sang-hyun Kim}
    \address{School of Mathematics, Korea Institute for Advanced Study (KIAS), Seoul, 02455, Korea}
    \email{skim.math@gmail.com}
    \urladdr{https://kimsh.kr}

    \author[T. Koberda]{Thomas Koberda}
    \address{Department of Mathematics, University of Virginia, Charlottesville, VA 22904-4137, USA}
    \email{thomas.koberda@gmail.com}
    \urladdr{https://sites.google.com/view/koberdat}

\begin{abstract}
We prove the following ``Torus Trichotomy Theorem'': every finitely generated subgroup of $\pls$ is either virtually free, a central extension of a cocompact Fuchsian group by a finite cyclic group, or a one-ended, non-acylindrically-hyperbolic group containing $\mathbb{Z}^2$.
In particular, we obtain the ``free or surface'' alternative for word-hyperbolic groups of PL homeomorphisms of the circle.
\end{abstract}
    \maketitle
    
\setcounter{tocdepth}{1}
\tableofcontents
    

\section{Introduction}\label{sec:intro}
In this paper we investigate obstructions for groups to act on the circle by piecewise linear homeomorphisms, and classify word-hyperbolic groups that act faithfully by piecewise linear homeomorphisms of the circle; throughout this paper, we write $\pls $ for the group of orientation-preserving piecewise linear homeomorphisms of the circle.

We will extensively use the following terms.
We emphasize our convention that virtually free groups are finitely generated, and virtual surface groups are word-hyperbolic.
\begin{defn}\label{defn:virtual}
\begin{enumerate}
    \item A group is \emph{virtually free} if a finite-index subgroup of it is a finitely generated free group.
    \item By a \emph{surface group}, we mean the fundamental group of a closed orientable hyperbolic surface.
    \item A group is a \emph{virtual surface group} if a finite-index subgroup of it is a surface group.
    \item By a \emph{Fuchsian group}, we mean a discrete subgroup $\Gamma$ of the orientation-preserving isometry group of the hyperbolic plane $\mathbb{H}^2$, identified with $\PSL_2(\mathbb{R})$.
    We say $\Gamma$ is cocompact if so is its action on $\mathbb{H}^2$.
    Note that a cocompact Fuchsian group is a virtual surface group.
\end{enumerate}
\end{defn}

We are especially interested in obstructions arising from commutativity, cohomology, and hyperbolicity. With this framing, we obtain a general structure theorem for finitely generated subgroups of $\pls$.
The definition of \emph{acylindrical hyperbolicity} is given in \cref{ss:acyl}.

\begin{thm}[Torus Trichotomy]\label{thm:main}
For each finitely generated subgroup $G$ of $\pls$, exactly one of the following holds.
\begin{enumerate}
    \item\label{p:free} $G$ is virtually free.
    \item\label{p:torus} For some finite cyclic group $C$ and for some cocompact Fuchsian group $Q$, we have a central extension
$$1\longrightarrow C\longrightarrow G\longrightarrow Q\longrightarrow 1.$$
\item\label{p:z2} $G$ is one-ended, not acylindrically hyperbolic, and contains $\mathbb{Z}^2$. 
\end{enumerate}
\end{thm}

Nicolás Matte Bon and Michele Triestino~\cite{MBT2026} recently proved that a finitely generated subgroup of the Thompson group $T$ is either virtually free or contains $\mathbb{Z}^2$.
As their result prohibits the second alternative of \cref{thm:main} in $T$, we obtain the following.
\begin{cor}\label{cor:thompson}
Every finitely generated non-virtually-free subgroup of the Thompson group $T$ is one-ended, non-acylindrically hyperbolic and contains $\bZ^2$.
\end{cor}

The proof of \cref{thm:main} combines dynamical features of PL actions such as expansiveness (\cref{s:1end}) and localization (\cref{s:free-prod}) with geometric group theoretic properties such as ends of groups, quasi-geodesic rays and virtual surface group realization. 
The glue holding fast these two approaches is cohomological; namely, the work of Ghys and Sergiescu~\cite{GS1987} on the \emph{jump cocycle} in $\pls$. 
The classical results of Hölder~\cite{Holder1901}, Denjoy~\cite{Denjoy1932} and Kopell~\cite{Kopell1970} also play a fundamental role.
The work of Brin--Squier~\cite{BS1985} on the non-existence of $F_2$ in $\PL_+(I)$ may be considered as a starting point of our investigations.

We prove a number of other results along the path toward establishing the main theorem. The following is one of the dynamical consequences:

\begin{cor}\label{cor:nonexp}
In alternative~\eqref{p:torus} of \cref{thm:main}, the action is minimal.
\end{cor}

It is a well-known result of Bleak and Salazar-D\'{\i}az~\cite{BS2013} that the group $\Z^2*\Z$ cannot embed in Thompson's group $V$, and so in particular cannot embed in Thompson's group $T$, the latter of which is identified with a group of piecewise linear homeomorphisms of the circle.

We generalize the Bleak--Salazar-D\'{\i}az result in the context of $\pls$ as follows. This theorem is also one of the most crucial dynamical-group-theoretic ingredients of this paper. We say a group is \emph{$\mathbb{Z}^2$--free} if it does not contain a copy of $\mathbb{Z}^2$ as a subgroup.

\begin{thm}\label{thm:free-prod}
If $H$ is a finitely generated, one-ended group, then $H\ast\mathbb{Z}$ does not embed into $\pls$.
\end{thm}

Recall that a \emph{right-angled Artin group (RAAG)} $A(\Gamma)$ is a group presented by a finite simplicial graph $\Gamma$ as
$$
A(\Gamma)=
\form{V(\Gamma)\mid
[v,w]=1\text{ for each edge }\{v,w\}\in E(\Gamma)}.
$$ 
See~\cite{Charney2007,koberda-survey} for general background.
The classification of possible isomorphism types of RAAG subgroups in diffeomorphism groups has played crucial roles in the study of smooth circle actions of finite-index subgroups of mapping class groups; see~\cite{BKK2014,BKK2019JEMS, KK2018JT,KK2020crit,KK21-book,KKR2021,KKR2024} for instance.
For us, ``commutativity obstructions'' are non-embeddability results arising from right-angled Artin groups. The following result characterizes the exact commutativity obstruction for piecewise linear homeomorphisms.
\begin{cor}\label{cor:raag}
A right-angled Artin group $G$ embeds into $\pls $ if and only if $G$ is free or free abelian.
\end{cor}

For a surface $S$ of genus $g$ and $n$ marked points, we write $c(S)=3g-3+n$ for its complexity.
\begin{cor}\label{cor:mcg}
    No finite-index subgroup of the mapping class group of a surface $S$ with $c(S)\geq 2$ can act faithfully by piecewise linear homeomorphisms of the circle.
\end{cor}

Corollary~\ref{cor:mcg} is of interest since it still appears to be unknown whether a finite-index subgroup of the mapping class group of a surface of complexity at least two can act $C^1$--smoothly on $S^1$, and groups of $C^1$ diffeomorphisms and groups of piecewise linear homeomorphisms often behave similarly. See~\cite{BKK2019JEMS,KKR2024,MannWolff20}.

A circle of ideas closely related to the central theme of this paper concerns manifold fundamental group actions on the circle. By Agol's resolution of the Virtual Fibering Conjecture~\cite{Agol2008,Agol2013}, fundamental groups of hyperbolic $3$--manifolds of finite volume can be realized as subgroups of right-angled Artin groups, up to passing to a finite-index subgroup. \cref{cor:raag} shows that most right-angled Artin groups cannot embed in $\pls $, though this does not rule out a faithful action by hyperbolic manifold groups; see~\cite{koberda-girth-2026} for related questions. Closed surface groups do occur as subgroups of $\pls $; see the work of Ghys~\cite{Ghys1987GV} and Minakawa~\cite{Minakawa-2001}. Right-angled Artin groups are linearly orderable and hence act faithfully by homeomorphisms (and even by $C^{\infty}$ diffeomorphisms, by~\cite{BKK2014}) of the real line, and so hyperbolic $3$--manifolds also (virtually) act faithfully in this way. Moreover, since right-angled Artin groups are residually torsion-free nilpotent, they act faithfully by $C^1$ diffeomorphisms on the circle; see~\cite{FF2003,Jorquera}. Despite the similarities between groups of diffeomorphisms and piecewise linear homeomorphisms of the circle, we have the following:

\begin{cor}\label{cor:hyp-manif}
    Let $G$ be virtually the fundamental group of a finite volume hyperbolic manifold of dimension at least three. Then $G$ admits no injective homomorphism into $\pls $.
\end{cor}

Besides word-hyperbolic groups, acylindrically hyperbolic groups and their actions are of central importance in geometric group theory, as well as for establishing Theorem~\ref{thm:main}. Here, by an \emph{acylindrically hyperbolic group}, we mean a group admitting a non-elementary acylindrical action on a Gromov hyperbolic space; see \cref{ss:acyl}. Mapping class groups of orientable hyperbolic surfaces~\cite{Bowditch2008} and irreducible right-angled Artin groups~\cite{KK2013b} are examples of acylindrically hyperbolic groups which are generally not (even relatively) hyperbolic.
We will establish the alternative~\eqref{p:z2} from a characterization of acylindrically hyperbolic subgroups of $\pls$. 

For the more general class of acylindrically hyperbolic groups, we have the following, which is crucial for establishing Theorem~\ref{thm:main}:
\begin{thm}\label{thm:acyl}
    Let $G\leq\pls $ be a non-virtually cyclic, finitely generated, acylindrically hyperbolic group. Then one of the following (not mutually exclusive) conclusions holds.
    \begin{enumerate}
        \item $G$ contains a free (possibly cyclic) group of finite index.
        \item $G$ contains a closed hyperbolic surface subgroup of finite index.
        \item $G_+$ splits nontrivially over a virtually cyclic subgroup.
    \end{enumerate}
\end{thm}

\subsection{AI Disclosure}
Many of the key ideas of the proofs for the main theorems were obtained from iterated guided conversations with OpenAI's ChatGPT. \emph{However, the writing of this manuscript was entirely done by the authors.} In particular, the authors independently verified and executed the given ideas.

\subsection*{Acknowledgments}
S.K.~thanks Xiaobing Sheng for helpful discussions and for pointing him to the reference~\cite{BS2013}.
L.D., M.E.~and S.K.~are supported by Mid-Career Researcher Program (RS-2023-00278510) through the National Research Foundation funded by the government of Korea.
M.E.~and S.K.~are also supported by KIAS Individual Grants (MG107601 and MG073602, respectively).
L.D.~and S.K.~are also supported by  KIAS--KAIST Joint Research Group Grant. T.K.~is partially supported by NSF grant DMS-2349814.

\section{PL maps and jump cocycles}
In this section, we lay the foundation of our main tool for the PL dynamics, namely the jump cocycle. Let us explicitly adopt the convention that
$$[u,v]:=u^{-1}v^{-1}uv.$$
\subsection{Basic facts on PL groups of one-manifolds}
A homeomorphism $f$ of the compact interval $I=[0,1]$ or circle $S^1=\R/\Z$ is called \emph{piecewise linear} if, away from a finite set of points, $f$ is given by an affine map of the form $f(x)=\alpha\cdot x+\beta$ for suitable constants $\alpha$ and $\beta$, and with $\alpha\neq 0$. The finite set $\BP(f)$ of points at which the derivative of $f$ is not continuous is called the set of \emph{breakpoints} of $f$. We write $\pls$ whenever orientation-preserving piecewise linear homeomorphisms are considered; these form a subgroup of index at most two.

For $x\in S^1$ and $r\ge0$, we denote the radius $r$--ball around $x$ as $$B(x,r):=[x-r,x+r];$$ this is a closed interval when $r<1/2$. 
An orbit of a point $x\in S^1$ under the action of a group $G\le\pls$ is written as $G.x$. We write the image of $x$ by $g\in\pls$ as $g.x$, $gx$ or $g(x)$, often interchangeably.
We denote by $\stab_G(x)$ the stabilizer of $x\in S^1$ in $G$.

For $f\in \pls$, the fixed point set $\Fix f$ is a finite union of (possibly degenerate) closed intervals. 
Hence, we obtain the following.
\begin{lem}\label{lem:boundary-finite}
For every $f\in\pls$, the set $\partial\Fix(f)$ is finite.
\end{lem}

We denote the \emph{(open) support} of $f\in\pls$ as
$$\supp f:=S^1\setminus\Fix f.$$
Each connected component of $\supp f$ is called a \emph{supporting interval}.
Whenever $\Fix f\ne\varnothing$, we have that $\Fix(f^n)=\Fix f$ and that $\supp f^n =\supp f$ for all $n\ne0$. For $G\le\pls$ and $x\in S^1$, we let $\stab_G(x)$ denote the stabilizer group of $x$ in $G$.

When we say $f$ is \emph{affine on an arc $J\subseteq S^1$}, we mean that the restriction of $f$ to $J$ coincides with a single affine map. 

We will often use the following theorem of Brin and Squier. Note that the second part is implicit in their work, and given explicitly in the work of Guba and Sapir.
\begin{thm}[\cite{BS1985,GS1999}]\label{thm:bs1985}
The group $\PL_+(I)$ does not contain 
a copy of $F_2$. Moreover, every nonabelian subgroup of $\PL_+(I)$ contains a copy of $\bZ\wr\bZ$.
\end{thm}

We recall the following structural result of Matucci on centralizer groups of point stabilizers for PL interval actions; Matucci's result can be viewed as a piecewise linear version of Kopell's Lemma~\cite{Kopell1970}.
\begin{lem}[\cite{Matucci-2010}]\label{lem:matucci2010}
If $f\in \PL_+(I)$ fixes no point in the interior, then the centralizer group of $f$ in $\PL_+(I)$ is infinite cyclic.
\end{lem}

\begin{defn}\label{defn:free}
We say a group $G$ is $H$--free if and only if $G$ does not contain a subgroup isomorphic to $H$.    
\end{defn}

Let us now make the following observation on point stabilizer groups of a $\mathbb{Z}^2$--free subgroup in $\pls$.
\begin{prop}\label{prop:point-stabilizer}
For all $\mathbb{Z}^2$--free subgroups  $G\le\pls$ and for all $x\in S^1$, the point stabilizer group $\stab_G(x)$ is trivial or infinite cyclic.
\end{prop}

\begin{proof}
Let $H:=\stab_G(x)$, considered as a subgroup of $\PL_+(I)$. 
By Theorem~\ref{thm:bs1985}, the group $H$ is abelian. Moreover, every  pair $g,h$ of nontrivial elements in $H$ must generate an infinite cyclic group, say $\form{u(g,h)}$. It follows that 
$$\supp g = \supp u(g,h)=\supp h.$$

Let us write the above support as a finite union of open intervals $\bigcup_{i=1}^n J_i$.
The restriction of $g$ to each $J_i$ is fixed point free. By \cref{lem:matucci2010}, the restriction of $H$ to each $J_i$ is cyclic. It follows that $H$ embeds in a finitely generated free abelian group $\bZ^n$. Since $H$ has no copy of $\bZ^2$, we have that $H$ must be cyclic, as desired.
\end{proof}

The \emph{rotation number} $\rot(f)$ of a homeomorphism $f\in\Homeo_+(S^1)$ is defined by taking an arbitrary lift $F$ to $\Homeo_+(\R)$, computing the limit $\lim_{n\to\infty}\frac{F^n(x)}{n}$ for $x\in\R$ arbitrary, and computing the result modulo one. For background on the rotation number, see~\cite{Ghys2001,Navas2011,KK21-book,KKM2019}, for instance. Some standard facts about the rotation number are that the rotation number is conjugacy--invariant in $\Homeo_+(S^1)$, that the rotation number of $f$ is rational if and only if $f$ has a periodic orbit, is zero if and only if $f$ has a fixed point, and is homogeneous in the sense that $\rot(f^n)=n\rot(f)$ for all $n\in \N$. 
Moreover, the rotation number restricts to a group homomorphism on each amenable subgroup of $\Homeo_+(S^1)$; see \cite{FF2001}. We will repeatedly use the following well-known consequence.
\begin{lem}[{See~\cite{Ghys2001}, for instance}]\label{lem:rot-fin}
If $G\le\Homeo_+(S^1)$ is a finite group, then the rotation number restricts to an injective group homomorphism from $G$ into $S^1$; in particular, $G$ is finite cyclic.
\end{lem}

Denjoy's theory of homeomorphisms investigates the rigidity of homeomorphisms $f\in\Homeo_+(S^1)$ with $\rot(f)\notin\mathbb Q$; see~\cite{Navas2011,KK21-book,athanassopoulos,Kim:aa}. Here, \emph{rigidity} of such a homeomorphism means that it is conjugate (in $\Homeo_+(S^1)$ generally) to an irrational rotation under certain analytic hypotheses. 
The following fact was established (in a more general form) by Herman:

\begin{thm}[{See~\cite{Herman1979}, also ~\cite[p. 38]{dMvS1993}}]\label{thm:herman}
Every element of $\pls$ having irrational rotation number is topologically conjugate to an irrational rotation of $S^1$ in $\Homeo_+(S^1)$.
\end{thm}
 
Since the centralizer of an irrational rotation is the rotation group in $\Homeo_+(S^1)$, we obtain the following consequence.
\begin{lem}[{cf.~\cite{FF2001}}]\label{lem:irrational}
If $f\in\pls$ has an irrational rotation number, then the centralizer of $f$ in $\Homeo_+(S^1)$ is topologically conjugate to the rotation group $\SO(2)$.
\end{lem}

The following classical theorem of Hölder characterizes groups freely acting on the circle; we state a more modern formulation due to Lodha--Matte Bon--Triestino.

\begin{thm}[\cite{Holder1901,LMBT2020}]\label{thm:hölder}
If a group $\Gamma$ acts faithfully and freely on the circle by orientation-preserving homeomorphisms, then the action is semi-conjugate to the action of a rotation group; in particular, $\Gamma$ is abelian.
\end{thm}

\subsection{Jump cocycle}
A fundamental object relating cohomology and dynamics is the \emph{jump cocycle},
defined by
$$\JJ(f,x):=\log \frac{D_+f(x)}{D_-f(x)}$$
for $f\in \pls$ and $x\in S^1$.
The cohomological properties of this cocycle have been studied extensively; see~\cite{Ghys1987,LMBT2020} for instance.
Recall our convention that $fg$ denotes $f\circ g$ for $f,g\in\pls$. The following is immediate from the definitions and the chain rule. 
\begin{lem} \label{lem:properties-jump-1}
    Let $f$ and $g$ be two elements in a finitely generated subgroup $G\leq \pls $, and let $x\in S^1$ be any point. Then, we have the following.
    \begin{enumerate}
    \item $\JJ(f,x)\ne0$ iff $x\in \BP(f)$.
    \item $\JJ(fg,x)=\JJ(f,g(x))+\JJ(g,x)$, and $\JJ(f^{-1},x)=-\JJ(f,f^{-1}x)$.        \item $\JJ({gfg^{-1}},gx)=\JJ(f,x)+\JJ(g,fx)-\JJ(g,x)$.
    \item $\BP(fg)\subseteq \BP(g)\cup g^{-1}\BP(f)$ 
    and $\BP(f^{-1})=f\BP(f)$.
    \end{enumerate}
\end{lem}
If $G\le\pls$ is generated by a finite set $\{s_1,\ldots,s_m\}$, then the set of \emph{group breakpoints} can be written as
$$\BP(G):=\bigcup_{g\in G}\BP(g)
\subseteq \bigcup_{i=1}^m G.\BP(s_i).$$

For $f\in \pls$, it is immediate from a telescoping sum to see that
$$\sum_{x\in S^1}\JJ(f,x)=0.$$
The following converse is also true.
\begin{lem}\label{lem:jump-converse}
Let $u:S^1\to\mathbb{R}$ be a finitely supported function such that $\sum_x u(x)=0$. 
Then there exists $f\in \pls$ such that $u(x)=\JJ(f,x)$. 
\end{lem}
\begin{proof}
Define $U:S^1\to\mathbb{R}$ as a piecewise constant map whose successive ratios are $e^{u(x)}$, and such that the integral is $1$. By integrating $U$, we obtain a suitable $f$.
\end{proof}

\subsection{Action cocycles and relative ends}
A \emph{real action cocycle} for a $G$-set $X$ is a map
$$c:G\times X\to\mathbb{R},$$
satisfying \[c(gh,x)=c(h,x)+c(g,hx).\]
We say $c$ is \emph{finitely supported} if for each $g\in G$ the function $x\mapsto c(g,x)$ vanishes at all but finitely many points $x$. 
From \cref{lem:properties-jump-1}, we have a finitely supported real action cocycle
$$\JJ(g,x),$$
which will be called the \emph{jump cocycle} for the action of $\pls$ on $S^1$.
Although this is the only cocycle that we need in this paper, it will be convenient for us to describe more general facts on action cocycles.
\begin{rem} 
We use the term action cocycle to distinguish it from the usual abelian-group-valued group 1-cocycle satisfying
$$\hat c(gh,x)=\hat c(g,x)+\hat c(h,g^{-1}x).$$
The two notions are interchangeable; for instance, one can define 
an action cocycle $c$ from a real-valued group 1-cocycle  by
$$c(g,x)=-\hat c(g,gx).$$
\end{rem}

In this subsection, we let $c$ be a finitely supported real action cocycle for some finitely generated group $G$ acting on $X$ with a fixed finite symmetric generating set $S$.
One can geometrically visualize $c$ as a ``discrete 1-form''. Namely, we consider the \emph{action groupoid} $G\ltimes X$, where vertices correspond to points in an orbit, and an arrow joins a point $x$ to $gx$ with a label $g\in G$. A $1$-cocycle $c$ then assigns a value $c(g,x)$ to this arrow, so that the cocycle relation becomes a 1-form (flow) condition of each directed diagram
$$
\xymatrix{
x\ar[dr]_{c(h,x)}\ar[rr]^{c(gh,x)} && ghx\\
& hx\ar[ur]_{c(g,hx)}
} 
$$
We also obtain the natural \emph{skew-product action}~\cite{Schmidt1977} of the group $G$ on the space $X\times \mathbb{R}$ defined as
$$ g(x,t) = (gx, t + c(g,x)).$$

Fixing a base point $x\in X$, 
we have the \emph{holonomy homomorphism}
$$\operatorname{hol}_x(c):\stab_G(x)\to \mathbb{R}$$ 
defined as
$\operatorname{hol}_x(c)(g):=c(g,x)$.
This value can also be obtained by considering an arbitrary directed loop in the action groupoid and adding the cocycle values over the arrows.
A 1-cocycle $c$ is \emph{exact} if it admits a \emph{potential function} $v$,
which is a real-valued (not necessarily finitely supported) function satisfying $$c(g,x)=v(x)-v(gx)$$ for all $g$. The holonomy is trivial for an exact cocycle.

Another way of visualizing group actions is to consider the \emph{orbit (Schreier) graph} of the orbit $Gx$; this is the locally finite simplicial graph whose vertex set is the orbit such that two vertices $gx$ and $sgx$ are joined for $g\in G$ and $s\in S$. 
If $Gx$ is a free orbit, then the orbit graph coincides with the (right) Cayley graph of $(G,S)$.  An important special case of an orbit graph is the \emph{coset Schreier graph}, which is the orbit graph of the transitive action of $G$ on the coset space $G/H$.
We let $e(G,H)$ denote the number of ends of this coset Schreier graph, which is called the number of \emph{relative ends}. 
Recall that a group is either 0-ended ($e(G,1)=0$, finite), 2-ended ($e(G,1)=2$, virtually cyclic), $\infty$-ended ($e(G,1)=\infty$, splitting over a finite group), or 1-ended~\cite{Freudenthal1931}.
Let us record the following for later use.
\begin{thm}[{\cite{Stallings1971}; see~\cite{Cornulier2019} for this formulation}]\label{thm:ends}
A finitely generated group $G$ has infinitely many ends if and only if one of the following happens for some finite subgroup $C\le G$:
\begin{enumerate}[(i)]
\item $G\cong A\ast_C B$ for some $A,B\le G$ with $[A:C]\ge 2<[B:C]$.
\item $G\cong A\ast_C$ for some $A\le G$ with $[A:C]\ge2$.
\end{enumerate}
\end{thm}

The relation between 1-cocycles and ends of Schreier graphs goes back at least to Houghton~\cite{Houghton1972,Houghton1978}. 

\begin{prop}\label{prop:holom}
Suppose $G$ is a finitely generated one-ended group acting on a set $X$ with cyclic stabilizers, 
and suppose $c:G\times X\to\mathbb{R}$ is a finitely supported real 1-cocycle. Then the following hold.
\begin{enumerate}
    \item\label{p:holom1} For each $x\in X$, the holonomy map
$$\operatorname{hol}_x(c):\stab_G(x)\to \mathbb{R}$$
vanishes; in other words, we have
$c(g,x)=0$ for all $g\in \stab_G(x)$.
\item\label{p:holom2} The cocycle $c$ is \emph{exact} on every $G$--orbit;
in other words, each $G$--orbit $\mathcal{O}$ admits a real-valued potential function $u_{\mathcal{O}}:\mathcal{O}\to \mathbb{R}$ such that 
$$c(g,y)=u_{\mathcal{O}}(y)-u_{\mathcal{O}}(gy)$$
for all $g\in G$ and $y\in \mathcal{O}$.
\item\label{p:holom3} If the Schreier graph of some $G$-orbit $\mathcal{O}$ is one-ended, then a potential function as in \eqref{p:holom2} can be chosen to be finitely supported. 
\item\label{p:holom4} If the action of $G$ on $X$ is free, then $c$ admits a finitely supported global potential $v$; in other words, a finitely supported $v:X\to\mathbb{R}$ satisfies
$$c(g,y)=v(y)-v(gy)$$
for all $g\in G$ and $y\in X$.
\end{enumerate}
\end{prop}

\begin{proof}
Let us fix a finite symmetric generating set $S$ of $G$.

\eqref{p:holom1}
By hypothesis $\stab_G(x)$ is cyclic. If $\stab_G(x)$ is finite then the homomorphism $\hol_x(c)$ is automatically trivial. Hence, we may suppose $\stab_G(x)=\form{a}\cong\mathbb{Z}$.
Assume for a contradiction that $\alpha:=c(a,x)\ne0$.
The skew product action 
$$G\curvearrowright X\times\mathbb{R}$$
is defined by $g(x,t)=(gx,t+c(g,x))$. We note that 
$g\in G$ stabilizes $(x,0)$ if and only if
$g=a^n\in \stab_G(x)$ and $c(g,x)=n\alpha=0$, which means that $g=1$.
In other words, the skew action is free and transitive on the orbit
$\mathcal{O}:=G.(x,0)$, whose orbit graph $\Gamma$ coincides with the (one-ended) Cayley graph. 

An edge in $\Gamma$ of the form $(y,t)~\sim (sy, t+c(s,y))$ changes the heights only if $y$ belongs to the finite set $\supp c(s)$. 
Let $\pi:X\times\mathbb{R}\to X$ denote the projection, 
and for each $y\in G.x$, we fix $g_y\in G$ satisfying $y=g_y(x)$. 
Then each fiber is a discrete set, written as
$$
\mathcal{O}\cap \pi^{-1}(y)=\{(y,c(g_y,x)+n\alpha)\;:\; n\in\mathbb{Z}\}.$$
It follows that $\mathcal{O}\cap\{t\ge0\}$ and $\mathcal{O}\cap\{t<0\}$ are both infinite sets of vertices, separated by finitely many edges in $\Gamma$. This contradicts the one-endedness, and hence, $\alpha=c(a,x)=0$ and $\hol_x(c)$ is trivial.   

\eqref{p:holom2} 
Fix a base $z\in\mathcal{O}$. By part \eqref{p:holom1}, we have a well-defined map $u=u_{\mathcal{O}}\co \mathcal{O}\to\mathbb{R}$ satisfying
$$u(gz):=-c(g,z)$$ for all $g\in G$.
It follows for $x=hz\in\mathcal{O}$ that
$$u(x)-u(gx)=c(gh,z)-c(h,z)=c(g,hz)=c(g,x).$$
This proves that $u$ is a potential.

\eqref{p:holom3} 
In general, if $f$ is a function defined on the vertex set of a locally finite one-ended graph, and if $f(x)=f(y)$ for all but finitely many edges $\{x,y\}$, then $f$ is constant outside finitely many vertices. We note the following.
In our case, a potential function $u_{\mathcal{O}}$ in part \eqref{p:holom2} changes values on an edge $\{x,sx\}$ of the orbit graph iff $c(s,x)$ is nontrivial. Since the Schreier graph is one-ended by hypothesis, and since $c$ is finitely supported, we see that $u$ is constant outside a finite set. Subtracting this constant from $u$, we obtain a new finitely supported potential.

\eqref{p:holom4} Consider the finite set $B:=\bigcup_{s\in S}\supp c(s)$. We have a finite collection of $G$-orbits $\mathcal{O}_1,\ldots,\mathcal{O}_m$ containing the set $G.B$.
We will define $u(x)=0$ for all $x$ not in $\bigcup_i\mathcal{O}_i$.
Since the action is free, the Schreier graph of each orbit coincides with the Cayley graph. Hence, we can apply part \eqref{p:holom3} and obtain a finitely supported potential $u_i:\mathcal{O}_i\to\mathbb{R}$ such that
$$c(g,y)=u_i(y)-u_i(gy)$$
for all $y\in\mathcal{O}_i$ and $g\in G$. 
Then the map $u:=\sum_i u_i\chi_{\mathcal{O}_i}$ is the desired global potential with the characteristic function $\chi_{\mathcal{O}_i}$.
\end{proof}


\begin{rem}\label{rem:perm-cohom}
Formally, we may consider the finite set
$$B:=\bigcup_{s\in S} \supp c(s).$$
We have a \emph{permutation $G$--module} \[V=V_{G,S}=\bigoplus_{GB}\bR\cong\bigoplus_{i=1}^m \bR[G/G_i];\] here, the left $G$--action on the left hand side is via the action of $G$ on $GB$, and on the right hand side by $G$ acting on its cosets. Each coset space $G/G_i$ is identified with an orbit $\mathcal{O}_i$ with $G_i=\stab_G(x)$ for a fixed base point $x$ of the orbit. Then $c$ becomes a group 1-cocycle with coefficients in $V$. The module $V$ can be viewed as the set of finitely supported $\bR$--valued functions on $GB$ which are zero at all but finitely many points, which is a subset of $\ell^2(X;\mathbb{R})$, as in the very neat framework of Lodha--Matte Bon--Triestino~\cite{LMT2019}.
In this viewpoint, we can interpret \cref{prop:holom} \cref{p:holom3} as saying that when every Schreier graph is one-ended, the cocycle $c$ is trivial in $H^1(G;V)$.
\end{rem}

\subsection{Balanced groups}\label{ss:balance}
Following~\cite{Wise2000,ButtonBalanced}, we say a group $G$ is \emph{balanced} if whenever $G$ contains $a,b$ such that $a$ has infinite order and $ba^mb^{-1}=a^n$ then $|m|=|n|$. 
Examples of balanced groups include semihyperbolic groups~\cite{BH1999} such as word-hyperbolic groups and CAT(0) groups, subgroup-separable groups~\cite{Wise2000}, virtually free-by-cyclic groups and groups linear over $\mathbb{Z}$~\cite{ButtonBalanced} and 3-manifold groups~\cite{Shalen2001}.
On the other hand, several manifold diffeomorphism groups are known to contain $\operatorname{BS}(1,n)$ with $n>1$ and hence are unbalanced; such a list includes $\Diff^\omega_+(S^1)$~\cite{BurslemWilkinson2004} (with a semiconjugacy classification in~\cite{GuelmanLiousse2011}), $\Diff^\infty_+(S^d)$~\cite{Asaoka2012}. The group $\pls$ is also balanced:
\begin{lem}\label{lem:balance}
The group $\pls$ is balanced.
\end{lem}
By definition, this means that if we have 
$gf^mg^{-1}=f^n$ for some $f,g\in\pls$ and for some integers $m,n$,
and if $f$ has infinite order, 
then  $|m|=|n|$.
\bpf[Proof of Lemma~\ref{lem:balance}]
Let $f,g,m,n$ be as in the above hypothesis.
If $\rot f$ is irrational, one can see the conclusion from the conjugacy--invariance and the homogeneity of $\rot$. 

So, we may assume $\rot f$ is rational;
by raising to a power, we may assume $\Fix f\ne\varnothing$.
Let $L$ be the (finite) set of the inward slopes of $f$; this means, we consider
$$
1\not\in L:=\{D^+f(\inf J),D^-f(\sup J)\;:\; J\in\pi_0(\supp f)\}.
$$
Since $g$ permutes $$\supp f=\supp f^m=\supp f^n,$$
and since affine conjugations do not change the original slopes, we see that 
$$\{a^m\;:\; a\in L\}=\{a^n\;:\; a\in L\}.$$
Comparing the largest logarithmic values from both sides, we see that $|m|=|n|$.
\epf
\begin{rem}
The group $\operatorname{BS}(1,-1)$ does embed into $\pls$. Indeed, set $r:=\Rot(1/2)$, and choose disjoint open intervals
$$A<B<rA<rB<A$$
on $S^1$. Choose $u,v\in \pls$ with supports $A$ and $B$ respectively,
and write 
$$a:=(rur^{-1})u^{-1},c:=(rvr^{-1})v.$$
Setting $b:=rc$, it is routine to verify (say, using the normal form of $\operatorname{BS}(1,-1)$) that $a,b$ generate the group
$$\form{a,b\;|\; bab^{-1}=a^{-1}}.$$
\end{rem}

\section{Dynamics of one-ended PL groups}\label{s:1end}
We now leverage one-endedness as an obstruction for embedding into $\pls$. 
\emph{In this section, we make the following standing assumptions:} 
\begin{hyp}\label{hyp:1end}
We write $H$ for a finitely generated, one-ended, $\mathbb{Z}^2$--free subgroup of $\pls$, with a fixed finite symmetric generating set $S$.
\end{hyp}
We note that $H$ has no finite orbit; for otherwise, the stabilizer of such a point would be a finite-index subgroup of $H$, which would imply by \cref{prop:point-stabilizer} that $H$ is virtually cyclic. Similarly, we note $H$ is non-abelian, and the action of $H$ is nonfree; see~\cref{thm:hölder}.

\subsection{Relative ends from potentials}\label{ss:rel-end}
We first prove under the above hypothesis that every nontrivial element $h$ of $H$ is \emph{topologically free}, meaning that $h$ does not pointwise fix a nontrivial interval.

\begin{lem}\label{lem:bp-fix}
For every $1\ne h\in H$, we have
$$\Fix h\cap\BP(h)=\varnothing.$$
Moreover, $\Fix h$ is finite, and $h$ is affine on some neighborhood of each point in $\Fix h$ with nontrivial slope.
\end{lem}

\bpf
This is straightforward from \cref{prop:holom}, the hypothesis of which is satisfied with $c$ being the jump cocycle $J$. If $x\in \Fix h$, then part~\eqref{p:holom1} of the proposition implies
$\JJ(h,x)=0$ and hence $x\not\in\BP(h)$. The map $h$ cannot fix an open interval, for otherwise one of the endpoints will violate the first conclusion.  The remaining conclusion is trivial.\epf

We produce more than one end relative to an infinite cyclic point stabilizer.
\begin{lem}\label{lem:cyc-stab}
For some point $p\in S^1$, the stabilizer group $\stab_H(p)$ is infinite cyclic and $e(H,\stab_H(p))>1$. 
\end{lem}
\bpf
Since the action of $H$ is nonfree, we can find a point $x\in S^1$ with some nontrivial $f\in H$ fixing it. Let $(a,b)$ denote a component of $\pi_0\supp f$; by inverting if necessary, we may assume $f(x)>x$ for $x\in (a,b)$ so that $$D^+f(a)>1>D^-f(b).$$ By telescoping the interior log derivative sums, we see 
$$
\sum_{x\in \BP(f)\cap(a,b)} \JJ(f,x)=\log D^-f(b)-\log D^+f(a)<0.$$
Since finitely many $f$--orbits cover the finite set $\BP(f)\cap(a,b)$, 
and since $(a,b)$ is $f$--invariant, 
we can find some $p\in\BP(f)\cap(a,b)$ such that
\begin{align}\label{eq:tfp}
T(f,p):=\sum_{k\in\mathbb{Z}} \JJ(f,f^kp)<0.
\end{align}
By \cref{prop:holom} \eqref{p:holom2}, the jump cocycle admits a potential $u$ on the orbit $H.p$. In particular, one has
$$u\left(f^{n}p\right)-u\left(f^{n+1}p\right)
= \JJ(f,f^np).$$
It follows that $u(f^{-m}p)$ and $u(f^np)$
 eventually become constants as $m,n\to\infty$, and that the difference will be $T(f,p)\ne0$. 
It follows that $u$ cannot be chosen to be finitely supported on the orbit $H.p$. 
By part \eqref{p:holom3} of the same proposition, it follows that the Schreier graph of $H.p$ is not one-ended; it is not 0-ended either, since the orbit is infinite. It follows that $e(H,\stab_H(p))>1$. Note that $\stab_H(p)\ne1$; for otherwise, $e(H)=e(H,1)>1$, contradicting the hypothesis. By \cref{prop:point-stabilizer}, the conclusion follows.
\epf

\subsection{Expandable tangent vectors}\label{ss:expandable}
We say a group $G\le\pls$ is \emph{normalized} if for all $$x\in \BP(G):=\bigcup_{g\in G}\BP(g),$$
we have $\stab_G(x)\ne\{1\}$. This can be guaranteed for $H$ up to conjugation:
\begin{lem}\label{lem:normalized}
There exists some $\phi\in\pls$ such that $\phi H\phi^{-1}$ is normalized.
\end{lem}
\bpf
Consider the following $H$--invariant ``free locus'':
$$ X_0:=\{ x\in S^1 \co \stab_H(x)=1\}.$$
Applying \cref{prop:holom} \cref{p:holom4}, we obtain a finitely supported potential function $v_0$ on $X_0$ such that 
$$\JJ(h,y)=v_0(y)-v_0(hy)$$
for all $h\in H$ and $y\in X_0$.
We extend $v_0$ to $S^1$ by zero outside $X_0$. 
Choose $p\in S^1$ such that $\stab_H(p)\ne1$. The function
$$v:=v_0-\left(\sum_{y\in S^1}v_0(y)\right)\delta_p,$$
with $\delta_p$ being the Dirac delta function, 
is a finitely supported real function that sums up to zero.
By \cref{lem:jump-converse}, we have some $\phi\in\pls$ such that
$$v(x)=\JJ(\phi,x)$$
for all $x\in S^1$.

Note for all $h\in H$ and $x\in S^1$ that
$$\JJ(\phi h \phi^{-1},\phi x)
=\JJ(h,x)+\JJ(\phi,hx)-\JJ(\phi,x).
$$
This is clear from the following diagram in the action groupoid $\pls\ltimes S^1$:
$$
\xymatrix{ 
x\ar[r]^\phi \ar[d]^h & \phi(x)\ar[d]^{\phi h \phi^{-1}} \\
h(x)\ar[r]^\phi &  \phi h(x).
}
$$

It only remains to prove that $\phi$ satisfies the desired conclusion.
Set $G:=\phi H\phi^{-1}\le\pls$, and suppose $\stab_G(x)$ is trivial. 
It suffices to note that $\JJ(\phi h\phi^{-1},x)=0$.
Let us write $y:=\phi^{-1}x\in X_0$ and see that
\begin{align*}
\JJ(\phi h\phi^{-1},\phi y)&=\JJ(h,y)+\JJ(\phi,hy)-\JJ(\phi,y)
\\ & =\left(v_0(y)-v_0(hy)\right)+ \left(v(hy)-v(y)\right)=0.
\end{align*}
This completes the proof.
\epf

An interval $I\subseteq S^1$ is an \emph{affine interval} for $h\in H$
if the restriction of $h$ to $I$ agrees with a single affine map.
If $I$ is a neighborhood of $x$, we also call it an \emph{affine
neighborhood of $x$} for $h$.
The conditions in \cref{hyp:1end} are invariant under $\pls$-conjugation,
and \cref{lem:normalized} allows us to pass to a normalized action.
We specify below which conclusions require normalization.

For any point $z$ with nontrivial stabilizer,
\cref{prop:point-stabilizer,lem:bp-fix} imply that $\stab_H(z)$ is infinite
cyclic and that each of its nontrivial elements is affine near $z$, with
slope different from $1$.
Exactly one of the two generators of $\stab_H(z)$ has slope greater than
$1$. We denote this generator by $g_z$ and call it the \emph{affine
repeller} at $z$. Thus
\[
D^+g_z(z)=D^-g_z(z)>1.
\]
In particular, when $H$ is normalized, every group breakpoint has an
affine repeller.

For $x\in S^1$, let $x^+$ and $x^-$ denote the positive and negative unit
tangent vectors at $x$, respectively.
We use $D^+h(x)$ and $D^-h(x)$ for the right- and left-hand slopes of $h$;
both are positive, and both equal the ordinary derivative $Dh(x)$ when
$x\notin\BP(h)$.
We say that $x^+$ is \emph{expandable} if $D^+h(x)>1$ for some $h\in H$,
and that $x^-$ is expandable if $D^-h(x)>1$ for some $h\in H$.
A nonzero tangent vector is expandable if the unit vector in the same
direction is expandable, and is \emph{non-expandable} otherwise.
This is the one-sided version of the notion of
expandability used in \cite{FK2012,DKN2009,Triestino2019JTop}.

Suppose that $H$ is normalized and that a unit tangent vector at $x$ is
non-expandable. Then $x$ is a free point: otherwise, the affine repeller
$g_x$ would expand both directions at $x$.
By normalization, a free point is not a breakpoint of any element of $H$.
Consequently, every $h\in H$ is differentiable at $x$, and
\[
Dh(x)\le 1 \qquad\text{for all }h\in H.
\]
In particular, both unit tangent vectors at $x$ are non-expandable.
Every point of the orbit $H.x$ is likewise free, so ordinary derivatives
of all elements of $H$ also exist at these orbit points.

In the next lemma, we adapt the distortion argument of
\cite{FK2012,DKN2009} to bound the distance from an orbit point to the
generator breakpoints. This yields one affine neighborhood common to
all elements of $H$, rather than a neighborhood depending on the element.
We will then prove that, for a normalized action, every unit tangent
vector is expandable; see \cref{lem:minimal}.

\begin{lem}\label{lem:nonexpandable}
Suppose that $H$ is normalized. If a unit tangent vector based at
$x\in S^1$ is non-expandable, then there exists $r>0$ such that every
element of $H$ is affine on the interval $J=(x-r,x+r)$.
\end{lem}
\bpf
Let $d$ denote the standard distance on $S^1=\mathbb{R}/\mathbb{Z}$, and
put
\[
\mathcal B:=\bigcup_{s\in S}\BP(s).
\]
For each $b\in\mathcal B$, let $g_b$ be its affine repeller, with slope
$\lambda_b>1$. Choose $R_b>0$ so small that $g_b$ is affine on
$B(b,R_b)$ and $\lambda_bR_b<1/5$.
Set
\[
c_0:=\min\left(
\left\{\frac15\right\}\cup
\left\{\frac{R_b}{\lambda_b}:b\in\mathcal B\right\}
\right)>0.
\]
We claim that, for every $h\in H$ and $b\in\mathcal B$,
\begin{equation}\label{eq:ell-s}
d(h(x),b)\ge c_0Dh(x).
\end{equation}

Fix $h\in H$ and $b\in\mathcal B$, and write
\[
y:=h(x),\qquad \mu:=Dh(x)\le 1,\qquad \ell:=d(y,b).
\]
As observed above, $y$ is free, whereas $b$ has nontrivial stabilizer.
Thus $\ell>0$.
If $\ell\ge R_b$, then
\[
\ell\ge R_b\ge c_0\mu.
\]
Suppose instead that $\ell<R_b$, and choose the largest integer $n\ge0$
such that
\[
\lambda_b^n\ell<R_b.
\]
For $0\le k\le n$, the point $g_b^k(y)$ remains in $B(b,R_b)$ and has
distance $\lambda_b^k\ell$ from $b$.
The chain rule therefore gives
\[
1\ge D(g_b^nh)(x)
  =D(g_b^n)(y)\,Dh(x)
  =\lambda_b^n\mu.
\]
The first inequality uses non-expandability at $x$, applied to the element
$g_b^nh\in H$.
See \cref{fig:ne} for the case in which $y$ lies to the left of $b$.
By maximality of $n$, we have $\lambda_b^{n+1}\ell\ge R_b$, and hence
\[
\ell\ge\frac{R_b}{\lambda_b^{n+1}}
     \ge\frac{R_b}{\lambda_b}\mu
     \ge c_0\mu.
\]
This proves \cref{eq:ell-s}.

Fix $0<r<c_0$, and let $J=(x-r,x+r)$.
We prove by induction on word length that every element of $H$ is affine
on $J$. The identity is affine on $J$.
Suppose that $h$ is affine on $J$, and let $s\in S$.
Writing $\mu=Dh(x)$, we have
\[
hJ=(h(x)-r\mu,h(x)+r\mu).
\]
By \cref{eq:ell-s}, every point of $\mathcal B$ has distance at least
$c_0\mu>r\mu$ from $h(x)$.
Thus $hJ$ contains no breakpoint of $s$, so $s$ is affine on $hJ$ and
$sh$ is affine on $J$.
Since $S$ generates $H$, the induction proves the lemma.
\epf

\begin{figure}[tbp]
\centering
\begin{tikzpicture}[>=stealth, line width=.6pt,
                    every node/.style={font=\small}]
  \draw (-4,0) -- (4,0);
  \foreach \x in {-3.1,-1.1,0}
    \draw (\x,-.08) -- (\x,.08);
  \draw[->] (-3.1,.25) -- (-3.1,.45) -- (-2.55,.45);
  \draw[->] (-1.1,.25) -- (-1.1,.45) -- (-.55,.45);
  \node at (-3.1,.85) {$g_b^n(y)$};
  \node at (-1.1,.85) {$y=h(x)$};
  \node[anchor=west] at (-0.3,.35) {$b\in\mathcal B$};
  \draw[fill=black] (0,0) circle (1.5pt);
  \draw[<->] (-3.1,1.4) --
    node[above] {$\lambda_b^n\ell<R_b$} (0,1.4);
  \draw[<->] (-1.1,-.35) -- node[below] {$\ell$} (0,-.35);
\end{tikzpicture}
\caption{The estimate in Lemma~\ref{lem:nonexpandable} when $y=h(x)$
lies to the left of $b$. The arrows indicate the positive unit directions.}
\label{fig:ne}
\end{figure}

\begin{lem}\label{lem:minimal}
The action of $H$ on $S^1$ is minimal.
If $H$ is normalized, then for each $x\in S^1$ there exists $h\in H$
admitting an affine neighborhood of $x$ such that $Dh(x)>1$.
In particular, for a normalized action, every unit tangent vector is
expandable.
\end{lem}
\begin{proof}
We first prove minimality, which is invariant under topological conjugacy.
By \cref{lem:normalized}, we may assume for this part of the proof that
$H$ is normalized.
Suppose for a contradiction that the action is not minimal.
As noted earlier, $H$ has no finite orbit. It therefore admits a minimal
invariant Cantor set $K$; see \cite{Ghys2001,Navas2011}, for instance.
Write
\[
S^1\setminus K=\bigcup_{i=1}^{\infty}J_i
\]
for the decomposition into complementary open intervals, or \emph{gaps}.

We first claim that the setwise stabilizer of each gap is trivial.
Suppose that a gap $J$ has nontrivial stabilizer.
Every element of $\stab_H(J)$ fixes both endpoints of $J$, so
\cref{prop:point-stabilizer} gives
\[
\stab_H(J)=\form{a}
\]
for some $1\ne a\in H$.
By \cref{lem:bp-fix}, the restriction of $a$ to $J$ is nontrivial.
Choose a supporting interval $U$ of $a$ contained in $J$.
For every $u\in U$ and every $1\ne h\in H$, either $hJ$ is disjoint
from $J$, in which case $h(u)\ne u$, or $hJ=J$.
In the latter case, $h=a^m$ for some $m\ne0$, and again $h(u)\ne u$,
since $a$ has no fixed point in $U$.
Thus every point of $U$ is free.
Normalization now implies that $a$ has no breakpoint in $U$, so $a$ is
affine on $U$.
Both endpoints of $U$ are fixed by $a$, and therefore its affine
restriction to $U$ is the identity, a contradiction.
This proves the claim.

It follows that every point of each closed gap $\cl J_i$ is free.
Indeed, an element fixing an interior point must stabilize $J_i$.
An element fixing an endpoint also stabilizes $J_i$, since it preserves
$K$ and the side of that endpoint containing $J_i$.
Consequently, for each $h\in H$, normalization implies that
$\cl J_i$ is disjoint from $\BP(h)$.
Since $\BP(h)$ is finite, $h$ is affine on a neighborhood of $\cl J_i$.

There is a longest gap, since the gaps are disjoint and have total length
at most $1$. Let $J_1=(p,q)$ be such a gap.
Every $h\in H$ permutes the gaps and is affine near $\cl J_1$, so
\[
Dh(p)=\frac{|hJ_1|}{|J_1|}\le1.
\]
Thus both unit tangent vectors at $p$ are non-expandable.
By \cref{lem:nonexpandable}, there is an open interval $J$ containing
$p$ on which every element of $H$ is affine.
Since $J$ meets $K$, minimality of the action on $K$ gives
\[
K\subseteq\bigcup_{h\in H}hJ.
\]

For arbitrary $g,h\in H$, we have
\[
g\restriction_{hJ}
  =\bigl((gh)\restriction_J\bigr)
    \circ\bigl(h\restriction_J\bigr)^{-1}.
\]
Both maps on the right are affine, so $g$ is affine on $hJ$.
These intervals cover $K$, and $g$ is also affine on a neighborhood of
each closed gap.
Hence $g$ has no breakpoint anywhere on $S^1$ and is therefore a rotation.
This would make $H$ a rotation group, contrary to its being nonabelian.
The action is thus minimal, and the same is true of the original action
before conjugation.

Now suppose that $H$ is normalized.
If there were a non-expandable unit tangent vector, then
\cref{lem:nonexpandable} would give a nonempty open interval $J$ on which
every element of $H$ is affine.
By minimality,
\[
S^1=\bigcup_{h\in H}hJ.
\]
The same composition formula shows that every element of $H$ is affine
on each $hJ$, and hence is a rotation.
This again contradicts the fact that $H$ is nonabelian.
Therefore every unit tangent vector is expandable.

Finally, fix $x\in S^1$.
If $\stab_H(x)$ is nontrivial, then its affine repeller $g_x$ has the
required affine neighborhood and derivative greater than $1$.
If $\stab_H(x)$ is trivial, then normalization implies that
$x\notin\BP(H)$, so every element of $H$ is affine near $x$.
Expandability of $x^+$ supplies an element $h\in H$ with $Dh(x)>1$.
\end{proof}

Lastly in this section, let us consider a sequence of local affine expansions with finitely many choices for each term. The existence of such a sequence will be used to construct an $H$-equivariant map from $S^1$ to $\partial H$; see Lemma~\ref{lem:ori}. This sequence is very similar to the \emph{expansion procedure}~\cite{DKN2009,FK2012} which was defined for analytic circle diffeomorphism groups.

\begin{defn}\label{defn:expansion}
Let $H\le\pls$.
\begin{enumerate}
\item 
Suppose $(E,c,C,r)$ is a quadruple consisting of a finite set $E\subseteq H$ and constants $1<c<C$, $0<r<1/(8C)$.
We say this quadruple is an \emph{expansion parameter} for $H$ if the slope of every $e\in E$ is at most $C$,
and if for each point $x\in S^1$, there exists some $e\in E$ such that $e$ is affine on $B(x,2r)$ with slope in $[c,C]$.
\item Let $(E,c,C,r)$ be a given expansion parameter for $H$. Suppose $x\in S^1$, and $e_1,e_2,\ldots\in E$. We set 
$a_0=1$, and $a_n:=e_n\cdots e_1$ for each $n\ge1$.
If $e_n$ is affine on $B(a_{n-1}x,2r)$ ($B(a_{n-1}x,r)$, respectively) with slope in $[c,C]$ for each $n$, 
then we say $\{a_n\}_{n\ge1}$ is a \emph{strong expansion sequence} (\emph{weak expansion sequence}, respectively) from the point $x$ for $(E,c,C,r)$.
\end{enumerate}
\end{defn}

A strong expansion sequence is automatically a weak expansion sequence as well. The key observation is the existence of such sequences.

\begin{prop}\label{expansion}
For a finitely generated, one-ended, $\mathbb{Z}^2$--free, normalized subgroup $H$ of $\pls$, all of the following hold.
\begin{enumerate}
    \item\label{p:expansion1} There exists an expansion parameter $(E,c,C,r)$ for $H$.
    \item\label{p:expansion2} If  $(E,c,C,r)$ is an expansion parameter, and if $x\in  S^1$, then there exists a strong expansion sequence $\{a_n\}$ from $x$. Furthermore, for each integer $N\ge1$ there exists $\delta>0$ such that whenever $y\in B(x,\delta)$, there exists a weak expansion sequence $\{b_n\}$ from $y$ such that $a_n=b_n$ for all $n\le N$. In this latter case, we have an affine bijection
$$a_n^{-1}:B(a_ny, r) \to B(y,r/Da_n(y))$$ with $Da_n(y)\ge c^n$
for every $n=1,2,\ldots,N$.
\end{enumerate}
\end{prop}
\bpf
\eqref{p:expansion1}
We first claim that for each point $x\in S^1$ there exists $h=h_x\in H$ such that $h$ admits an affine neighborhood $J_x$ at $x$ with $Dh(x)>1$. If $x\in\BP(H)$, then the normalization condition implies that there is an affine repeller $h\in H$ at $x$. 
Suppose $x$ is not a group breakpoint. Since every unit vector is expandable by \cref{lem:minimal}, there exists an $h\in H$ such that $D^+h(x)>1$. Since $x$ is not a group breakpoint, this map $h$ admits an affine neighborhood of $x$. 

Pick a finite subcover $\{J_1,\ldots,J_m\}$ of $\{J_x\}_{x\in S^1}$ and let $e_i$ correspond to each $J_i$.
Let $\lambda_i>1$ be the slope of $e_i$ on $J_i$.
Choose $c$ and $C$ such that
$$1<c<\min_i\lambda_i,\qquad
C>\max\left\{c,\max_{1\le i\le m}\sup_{y\in S^1}
\max\{D^+e_i(y),D^-e_i(y)\}\right\}.$$
Choose $0<r<1/(8C)$ sufficiently small that every ball $B(x,2r)$ is contained in some $J_i$.
Then $E=\{e_1,\ldots,e_m\}$ and $c,C,r$ form an expansion parameter.

\eqref{p:expansion2} The existence of a strong expansion sequence is clear from induction.
For the latter part, choose $y$ sufficiently close to $x$ that
$B(a_i y,r)\subseteq B(a_i x,2r)$ for $i=0,\ldots,N-1$.
The first $N$ terms then give a weak expansion sequence from $y$, which can be continued by induction.
Each inverse step is affine on the radius-$r$ ball about its image and contracts that ball by a factor at most $1/c$.
Composing these inverse steps gives the asserted affine bijections.
\epf

\subsection{Affine discreteness and cyclic commensurators}\label{ss:commensurator}
We continue under \cref{hyp:1end} for $H\le\pls$; we do not impose the assumption that $H$ is normalized here. 
We will establish that $H$ is \emph{locally affine discrete} in the following sense; see~\cite{AAMMT2024} for the use of local discreteness in a more general setting.

\begin{lem}\label{local-discrete}
For every nonempty open interval $U\sse S^1$, there does not exist a sequence $\{g_j\}_{j\ge1}\subseteq H\setminus\{1\}$ such that each $g_j$ is affine on $U$ and 
such that $g_j\restriction_U\to 1\restriction_U$ in the uniform topology.
\end{lem}

Before proving \cref{local-discrete}, we make an observation about commuting conjugate elements.
We say two subgroups $K_1,K_2$ of $L$ are \emph{commensurable (in $L$)} if $K_1\cap K_2$ has finite index in each of the subgroups. The \emph{commensurator} of a subgroup $K$ of a group $L$ is the group defined as
$$ 
\Comm_L(K):=\{ g\in L \;|\; K\text{ and }gKg^{-1}\text{ are commensurable in }L\}.
$$

\begin{lem}\label{ab-vz}
Let $r,h\in H$ be such that $r$ has infinite order, $h\ne1$,
$\Fix h\ne\varnothing$, and $[r,hrh^{-1}]=1$.
Then $\rot r$ is rational, and $h$ belongs to the group
$\Comm_H(\form{r})$, which is virtually cyclic.
\end{lem}
\bpf
Setting $z:=hrh^{-1}$, we obtain from $\mathbb{Z}^2$--freeness that
$K:=\form{r,z}$ is of the form $\mathbb{Z}\oplus\mathbb{Z}/m$.
Killing torsion, we have $r^N=z^M$ for some nonzero multiples $N,M$ of $m$.
In particular, $h\in\Comm_H(\form{r})$.
Since $r$ has infinite order, \cref{lem:balance} implies that $|N|=|M|$.

Assume for contradiction that $\rot r\notin\mathbb{Q}$.
By \cref{lem:irrational}, the commuting elements $r$ and $z$ are
conjugate by the same element into the rotation group $\SO(2)$.
Since $\rot z=\rot r$, these two rotations coincide, so $z=r$.
Thus $h$ centralizes $r$ and is also conjugate into $\SO(2)$.
Since $h$ fixes a point, this implies $h=1$, a contradiction.

Hence, $\rot r$ is rational.
Choose $q>0$ such that $F:=\Fix r^q$ is nonempty.
By \cref{lem:bp-fix}, $F$ is finite, and $\Fix r^{kq}=F$
for every nonzero integer $k$.
If $u\in C:=\Comm_H(\form{r})$, then
$ur^a u^{-1}=r^b$ for some nonzero integers $a,b$.
Taking $q$th powers shows that $uF=F$.
For $p\in F$, the group $C\cap\stab_H(p)$ therefore has finite index
in $C$ and is cyclic by \cref{prop:point-stabilizer}.
This verifies the conclusion.
\epf

\bpf[Proof of \cref{local-discrete}]
Assume $H$ is not locally affine discrete. We have an infinite sequence $\{g_j\}\sse H\setminus\{1\}$ admitting a fixed affine interval $U$ such that $g_j\restriction_U$ converges to $1\restriction_U$. Since $H$ acts nonfreely, there exists $h\in H\setminus\{1\}$ with a fixed point $p$. By minimality (\cref{lem:minimal}), we may assume $p\in U$. Using \cref{lem:bp-fix} and shrinking $U$, we may assume that $h$ is affine on $U$. 
Define
$$r_j:=[h,g_j]=h^{-1}g_j^{-1}hg_j,\quad z_j:=h r_j h^{-1}.$$
After discarding finitely many terms, choose an open neighborhood $V$ of $p$
compactly contained in $U$ such that all the compositions below are computed
using the affine restrictions of $h$ and $g_j$.
Since $[\Aff_+(\mathbb{R}),\Aff_+(\mathbb{R})]$ consists of translations,
the maps $r_j$ and $z_j$ act by translations on $V$ and converge uniformly
to the identity there.
Their commutator is therefore the identity on a smaller interval about $p$.
By \cref{lem:bp-fix}, we have $[r_j,z_j]=1$ in $\pls$.

Set $F:=\Fix h$, a nonempty finite set by \cref{lem:bp-fix}.
We claim that every $r_j$ preserves $F$.
If $r_j$ has finite order, then the commuting elements $r_j$ and $z_j$
generate a finite subgroup. Since $\rot r_j=\rot z_j$,
\cref{lem:rot-fin} gives $r_j=z_j$.
Thus $r_j$ commutes with $h$ and preserves $F$.
If $r_j$ has infinite order, \cref{ab-vz} shows that
$\form{h,r_j}$ is virtually cyclic.
Since $h$ has infinite order, we can choose $N_j\ge1$ such that
$\form{h^{N_j}}$ is a finite-index normal subgroup of $\form{h,r_j}$.
It follows that $r_j$ preserves $\Fix h^{N_j}=F$.

Since $r_j(p)\to p$ and $F$ is finite, we have $r_j(p)=p$ for all sufficiently large $j$.
As $r_j$ is a translation on $V$, it is the identity on $V$,
and hence is the identity everywhere by \cref{lem:bp-fix}.
Thus $g_j$ commutes with $h$ and preserves $F$ for all sufficiently large $j$.
The convergence $g_j(p)\to p$ now gives $g_j(p)=p$ for all sufficiently large $j$.
By \cref{prop:point-stabilizer,lem:bp-fix}, we have
$\stab_H(p)=\form{t}$ with $Dt(p)=\tau\ne1$.
Write $g_j=t^{k_j}$. Since $g_j$ is affine on the fixed interval $U$,
its uniform convergence to the identity gives
$$\tau^{k_j}=Dg_j(p)\longrightarrow1.$$
Thus $k_j=0$ for all sufficiently large $j$, contradicting $g_j\ne1$.
\epf

We now establish the finiteness of elements with given affine restrictions $(\ell,a,b)$ in the following sense:

\begin{lem}\label{bounded-affine}
For each $\ell>0$ and $0<a\le b<\infty$, let us denote by $\mathcal{F}(\ell,a,b)$
the set of $h\in H$ such that $h$ is affine on some interval (possibly depending on $h$) of length $\ell$, with slope in $[a,b]$.
Then $\mathcal{F}(\ell,a,b)$ is finite.
\end{lem}
\bpf
Without loss of generality, we set $\ell<1/4$.
Assume for a contradiction that $\{h_j\}\sse \mathcal{F}(\ell,a,b)$ is an infinite distinct collection such that $h_j$ is affine on $[x_j,x_j+\ell]$
with the given slope condition.
By taking a subsequence of $\{h_j\}$, 
we obtain an open neighborhood $U$ of a point $p$ on which each $h_j$ can be written in the chart with $p=0$ as 
$$h_j(x)=A_jx+B_j$$
such that $A_j\to A\in[a,b]$. We may further assume that $h_j(p)$ converges, meaning that $B_j\to B$. Choosing a sufficiently small neighborhood $V$ compactly contained in $U$, we have that $g_j:=h_{j+1}^{-1}h_j$ is affine and converges to the identity on $V$. This violates the local-discreteness in \cref{local-discrete}.
\epf

We also note the following extension of \cref{ab-vz}.

\begin{lem}\label{lem:cyc-comm}
For every infinite order element $f\in H$, the group $\Comm_H(\form{f})$ is virtually cyclic.
\end{lem}

\begin{proof}
Every element $h\in C:=\Comm_H(\form{f})$ satisfies
$$hf^n h^{-1}=f^{\varepsilon n}$$
for some $n>0$ and $\varepsilon\in\{1,-1\}$, by \cref{lem:balance}.
Suppose first that $\rot f\in\bQ$, and choose $N>0$ such that $\rot f^N=0$.
The finite set $F:=\Fix f^N$ is invariant under $C$, since taking $N$th powers
in the displayed identity gives
$$hF=h\Fix f^{nN}=\Fix f^{\varepsilon nN}=F.$$
For any $p\in F$, the cyclic group $C\cap\stab_H(p)$ therefore has finite index
in $C$, by \cref{prop:point-stabilizer}.

Let $\rot f\notin\bQ$.
Conjugacy invariance of rotation number excludes $\varepsilon=-1$:
otherwise $2n\rot f=0$ in $\mathbb{R}/\mathbb{Z}$.
Thus every $h\in C$ centralizes a nonzero power $f^n$.
By \cref{lem:irrational}, the centralizer of $f^n$ in $\Homeo_+(S^1)$
is conjugate to the rotation group and is abelian.
Since it contains both $h$ and $f$, these elements commute.
Consequently, $C=Z_H(f)$, and $C$ is conjugate into $\SO(2)$.
In particular, $C$ acts freely and is abelian.

As in \cref{eq:tfp}, let us set
$$T(f,x):=\sum_{k\in\mathbb{Z}}\JJ(f,f^kx),$$
which is a finite sum for each $x\in S^1$.
For each $a\in C$, we note that
\begin{align*}
T(f,ax)-T(f,x)
&=\sum_{k\in\mathbb{Z}}\left(\JJ(f,af^kx)-\JJ(f,f^kx)\right)\\
&=\sum_{k\in\mathbb{Z}}\left(\JJ(a,f^{k+1}x)-\JJ(a,f^kx)\right)=0.
\end{align*}
Here, the last equality comes from telescoping, using the fact that $\BP(a)$ is finite.
Hence, $T(f,\cdot)$ is $C$--invariant.

Setting $b:=|\BP(f)|$, the desired conclusion will be straightforward from the claim that
 \begin{align}\label{cf}
 [C:\form{f}]\le b.
 \end{align}

It remains for us to prove \cref{cf}. Let us assume for contradiction that there are $b+1$ distinct coset representatives $$a_1, \ldots,a_{b+1}$$ 
of $C/\form{f}$. Since $C$ acts freely, for every $x$ the $f$-orbits
$$a_1\form{f}x, a_2\form{f}x, \ldots, a_{b+1}\form{f}x$$
are disjoint.
For each $x\in S^1$ there exists at least one $i=i(x)$ such that 
$a_i\form{f}x$ is disjoint from $\BP(f)$; in this case, we have
$$T(f,a_ix)=0.$$ 
It follows from the $C$--invariance of $T(f,\cdot)$ that $T(f,x)=0$ for all $x\in S^1$. 

We can pick $N>0$ such that for every $y\in\BP(f)$, we have $\JJ(f,f^ky)=0$ for all $|k|> N$. 
Then for all $m> N$, we have that
$$\sum_{k\ge0} \JJ(f,f^{k+m}y)=\sum_{k\ge0} 0=0,$$
and that
$$\sum_{k\ge0} \JJ(f,f^{k-m}y)=\sum_{k\in\mathbb{Z}}\JJ(f,f^ky)=0.$$
We may define a finite sum
$$v(x):=\sum_{k\ge0} \JJ(f,f^kx), $$
which is supported in a finite set
$$E:=\bigcup_{m=-N}^N f^m\BP(f).$$
Note from the cocycle relation that
$$v(x)-v(f^n x)=\sum_{k=0}^{n-1}\JJ(f,f^kx)=\JJ(f^n,x).$$
In particular, we would have
$$\BP(f^n)\sse E\cup f^{-n}E.$$

Choose an increasing sequence $\{n_j\}$ such that $f^{n_j}$ uniformly converges to $1$. We may then assume that $E$ and $\bigcup_{j\ge1} f^{-n_j}E$ are contained in some proper compact interval of $S^1$, so that a nondegenerate open interval $V$ misses that  compact interval. Then the maps $f^{n_j}$ are all affine on $V$; the uniform convergence condition also implies that their restrictions to $V$ converge to $1\restriction_V$. This contradicts \cref{local-discrete}. This completes the proof of the claim that $[C:\form{f}]\le b$.
\end{proof}

In the above proof, we note the trivial bound that 
$$1\le [C:\form{f}]\le b.$$
As this means $\BP(f)\ne\varnothing$, we obtain the following.
\begin{prop}\label{prop:irr-rot}
There is no irrational rotation in $H$.
\end{prop}

\section{Free products in $\pls$}\label{s:free-prod}
We now establish \cref{thm:free-prod} in this section, which will be one of the main dynamical group theoretic ingredients in the remainder of this paper. Generally speaking, a common technique in~\cite{BS1985,BKK2019JEMS,KK2018JT,KK2020crit} for proving non-embeddability results 
$$G\to \Homeo(X)$$
is to rely on \emph{localization}. Namely, one finds a pair of elements $u,v$ in the ambient transformation group $\Homeo(X)$ such that $u$ is compactly supported in a proper open set $U$, and such that $v$ moves $U$ off itself. One then obtains a commutativity result
$$[u,vuv^{-1}]=1,$$
since $\supp (vuv^{-1})=v\supp u$ is disjoint from $U$. When dealing with multiple components of $U$, one chooses a sequence of suitable choices of $v$, resulting in a sequence of iterated commutators $u_{i+1}=[u_i,v_i u_i v_i^{-1}]$, terminating at some finite index $u_{N+1}=1$. If one establishes that these trivial elements arise from nontrivial words in $G$, one obtains non-injectivity of the map $G\to\Homeo(X)$. This idea goes at least back to Zassenhaus~\cite{Zassenhaus1938}.

The uniform distance between $f,g\in\Homeo(S^1)$ will be denoted by $d_0(f,g)$.
Consider the word 
$$w_0:=\left[\left[a,b\right],\left[a^{-1},b^{-1}\right]\right]$$
in $F_2=\form{a,b}$.
The localization criterion below will be useful for us. 

\begin{lem}\label{localize}
Let $K\subseteq U\subset S^1$ be such that $K$ is compact and $U$ is open. Then there exists $\epsilon=\epsilon(K,U)>0$ such that the following holds:
\begin{quote}
if $r,s\in \pls$ satisfy that $\BP(q)\subseteq K$ and $d_0(q,1)<\epsilon$ for all $q\in Q:=\{r^{\pm1},s^{\pm1}\}$, then 
$$\supp w_0(r,s)\subseteq U.$$
\end{quote}
\end{lem}

\bpf
We may assume
$$\varnothing\ne K\sse U\subsetneq S^1,$$
for otherwise the proof is trivial. 
Pick an arbitrary $x\in S^1\setminus U$, and choose a sufficiently small affine neighborhood $J$ of $x$ for the elements in $Q$. Assuming the elements $r,s$ are chosen with sufficiently small $\epsilon>0$, we can require that all the suffixes $w_i(r,s)$ of the length-16 word $w_0(r,s)$ are affine on $J$. Since the commutator of two elements in $\Aff_+(\mathbb{R})$ is a translation, it follows that $w_0(r,s)$ fixes $x$.
\epf

Let us now prove a special case of \cref{thm:free-prod}.
\begin{lem}\label{lem:z2z}
The group $\mathbb{Z}^2\ast\mathbb{Z}$ does not embed into $\pls$.
\end{lem}
\bpf
Let us assume for a contradiction that 
$$G:=\mathbb{Z}^2\ast \mathbb{Z}=\form{a,b,t\;|\; [a,b]=1}$$
is realized as a subgroup of $\pls$.

We first consider the case when $A:=\form{a,b}\cong\mathbb{Z}^2$ has a finite orbit. 
After raising to powers, we may assume that $A$ has a common fixed point $p$; this is due to the fact that $b$ permutes the finite set $\partial\Fix a$ (\cref{lem:boundary-finite}). By restriction, we obtain an induced action of $A$ on $\PL_+(I)$. Since $\PL_+(I)$ does not contain $F_2$ by~\cite{BS1985}, we have that $t(p)\notin\Fix A$.

We claim in this case that for every point $y\in\supp A:=S^1\setminus\Fix A$,
there exists some $1\ne h\in A$ such that $h(y)=y$.
Let $J$ be the component of $\supp A$ containing $y$.
Every element of $A$ preserves $J$.
Choose $1\ne g\in A$. If $g\restriction_J=1$, then $g$ already proves the claim.
Otherwise, if $\Fix g\cap J\ne\varnothing$, then
$\partial\Fix g\cap J$ is a nonempty finite $A$--invariant set.
Since $A$ preserves the order on $J$, it fixes every point of this set,
contradicting $J\cap\Fix A=\varnothing$.
Thus $g$ has no fixed point in $J$.
As $A$ is abelian, \cref{lem:matucci2010} implies that $A\restriction_J$ is cyclic.
The restriction homomorphism $A\to A\restriction_J$ therefore has a nontrivial kernel,
which proves the claim.

Let $g\in A$ be a nontrivial element fixing $t(p)$.
Since $g(p)=p$, both $g$ and $tgt^{-1}$ belong to $\stab_G(t(p))$.
But this contradicts~\cite{BS1985} again, since an elementary combinatorial
group theoretic observation implies that $\form{g,tgt^{-1}}\cong F_2$.

We may now assume that $A$ has no finite orbit. If some nontrivial element $g$ of $A$ fixes a point, then $A$ would again admit a finite orbit $\partial\Fix g$ (\cref{lem:boundary-finite}). This implies that every nontrivial element of $A$ has irrational rotation number and the action of $A$ is free. By \cref{lem:irrational}, the group $A$ is topologically conjugate to a rotation group.
In particular, the action of $A$ is free and \emph{recurrent}; there is a sequence of nontrivial elements that uniformly converges to the identity.

By \cref{prop:holom} \eqref{p:holom4}, we have a  potential $v$ with a finite support $E$ for the jump function $\mathcal{J}$. In particular, every $a\in A$ satisfies
$$\BP(a)\sse E\cup a^{-1}E.$$
Let us consider the finite set
$$F:=E\cup t(E)\cup t\BP(t)\cup \{p\}$$
for an arbitrary $p\in S^1$.
Since the action of $A$ is free and since $|A|=\infty$, we can find some $d\in A$ moving $F$ off itself.
Choose a compact $K$ and an open $U$ such that 
$$F\sse \inte K\sse K\sse U\subsetneq S^1$$
and such that $U\cap dU=\varnothing$. 

Let $\delta>0$ be small,
and choose a nontrivial $a\in A$ such that $d_0(a,1)<\delta$ 
and $d_0(tat^{-1},1)<\delta$.
For $\delta$ sufficiently small, we have
$$\BP(a)\sse E\cup a^{-1}E \sse F\cup a^{-1}F \sse\inte K.$$
Also, we note
\begin{align*}
\BP(ta^{\pm1}t^{-1})
&\sse 
t\BP(t)\cup t\BP(a^{\pm1})\cup ta^{\mp1}\BP(t)
\\
&\sse
F\cup  tE\cup ta^{\mp1}t^{-1}\cdot tE
\cup
ta^{\mp1}t^{-1} F\sse\inte K.
\end{align*}
By further reducing $\delta>0$ if necessary, we obtain from \cref{localize} that 
$\supp w_0(a,tat^{-1})\sse U$.
This implies that
$$\left[w_0\left(a,tat^{-1}\right),dw_0\left(a,tat^{-1}\right)d^{-1}\right]$$
is a nontrivial word in $\mathbb{Z}^2\ast\bZ$ that maps to the trivial word in $\pls$. This is a contradiction.
\epf

We note below a rather technical lemma, which provides another form of localization. We suggest that the reader skip its proof until its use in the proof of \cref{thm:free-prod}.

\begin{lem}\label{certi}
Let $\epsilon,\lambda,\beta$ and $C_0$ be positive constants such that $\lambda<1$ and 
\begin{align}\label{eq:certi}
\frac{C_0}{\epsilon}<\frac{\lambda^2}{(2+\lambda)(1-\lambda)}.
\end{align}
Suppose $f,b\in \pls$ satisfy
$$f(x)=\lambda x,\quad b(x)=\beta(x+C_0)$$ for $0\le x\le \epsilon$,
and that $b(0,\epsilon)\subseteq (0,\epsilon)$. 
Then, after setting 
$$d:=(1/\lambda-1)C_0,\quad
s:=[f,b],\quad
w:=[s,fsf^{-1}],$$
we have that
$$\varnothing\ne J:=\left((1+\lambda)d,\lambda \epsilon-d\right)\subseteq \Fix w\cap \Fix \left(sws^{-1}\right).$$
\end{lem}
\bpf
The nonemptiness of $J$ is precisely equivalent to the hypothesis~\eqref{eq:certi}, which also implies $d<\epsilon$. We  see that 
$$fb(d,\epsilon)=(\lambda\beta(d+C_0),\lambda\beta(\epsilon+C_0))\subseteq(\beta C_0,\beta(\epsilon+C_0))=b(0,\epsilon).$$
It follows that each suffix of the word $s=f^{-1}b^{-1}fb$ is an affine map on $(d,\epsilon)$. By affine compositions, we easily see that
$$s:(d,\epsilon)\to (0,\epsilon-d), \quad s(x) = x-d.$$
Setting $\hat s:=fsf^{-1}$, we similarly obtain
$$\hat s:(\lambda d,\lambda\epsilon)\to (0,\lambda(\epsilon-d)),\quad
\hat s(x)=x-\lambda d.$$
Considering suffixes of $w=s^{-1}\hat s^{-1}s\hat s$, 
it is routine to check that $w(x)$ can be computed for $x\in J_1:=((1+\lambda)d,\lambda\epsilon)
\sse
(\lambda d,\lambda \epsilon)$ as
$$x\mapsto x-\lambda d\mapsto x-(1+\lambda)d
\mapsto x-d\mapsto x.$$
Here, we only need to verify that 
$$x-\lambda d\in (d,\epsilon),\quad
x-(1+\lambda)d \in (0,\lambda \epsilon-\lambda d),\quad
x-d \in (0, \epsilon-d)
$$
for such an $x$. By conjugation, we obtain
$sws^{-1}(x)= \hat s^{-1}s\hat s s^{-1}(x)=x$  
for $x\in J_2:=s((1+\lambda)d,\lambda\epsilon)
=(\lambda d, \lambda\epsilon-d)$. 
The conclusion follows from $J=J_1\cap J_2$.
\epf

Let us continue the proof of \cref{thm:free-prod}.

\begin{proof}[Proof of \cref{thm:free-prod}]
Suppose for contradiction that $H\ast\bZ=H\ast\form{t}$ is realized as a subgroup of $\pls$.
By \cref{lem:z2z} and \cref{lem:normalized}, we may assume that $H$ is finitely generated, one-ended, $\mathbb{Z}^2$--free and normalized.
We have noted in \cref{s:1end} that $H$ is nonabelian, nonfree, minimal with every unit vector being expandable; see \cref{lem:minimal}. Note that $H$ is not a torsion group, so that it contains a copy of $\mathbb{Z}$.
We can find embeddings
$$H\ast\mathbb{Z}\hookrightarrow H\ast H\hookrightarrow H\ast K.$$
Hence, it suffices to consider the case that $K=\form{t}\cong \mathbb{Z}$.

We claim that there exists an element $c\in H$ such that $f:=ct$ can be parametrized at some point $p\in S^1$ with $\lambda\in(0,1)$ as 
$$f(u)=\lambda (u-p)+p$$
for $u\in[p,p+r_1]$, up to conjugating by a reflection. Note that we allow $p$ to be a breakpoint of $f$.
For this, choose some $q\in S^1$ such that $\stab_H(q)=\form{c_1}$ for an infinite order element $c_1$. By \cref{lem:bp-fix}, the map $c_1$ admits an open affine neighborhood $J$ of $q$; assume $c_1$ is attracting, by taking its inverse if necessary.
By minimality, we can find some $c_2\in H$ and some compact neighborhood $V$ of $q$ such that $c_2t(V)$ is  contained in $J$. Since $c_1$ is attracting on $J$, we can find some $n\gg0$ such that $f:=c_1^nc_2t$ maps $V$ into the interior of $V$,
and such that the slopes of $f\restriction_V$ are less than $1$. It follows that 
$f(p)=p$ for some $p\in V$ with $f$ being attractive on one side. Since $$\form{f,\stab_H(p)}\cong\mathbb{Z}\ast \stab_H(p),$$
it follows from \cref{thm:bs1985} that $\stab_H(p)$ is trivial.
The claim is thus proved.

Assuming $p=0$, we now claim that 
for some $N\gg0$ and $\hat c\in H$, 
the hypotheses of \cref{certi} are satisfied by $f$, $b:=f^N\hat c$ and $\lambda$,
along with some constants $\beta,C$ and $\epsilon$.
To see this, let us choose an expansion parameter $(E,c,C,r)$ from \cref{expansion}. 
We have a strong expansion sequence $\{a_i\}$ from $0$, as in \cref{defn:expansion}. Since $H.0$ is a free orbit, and since $d_i:=Da_i(0)\ge c^i$ is strictly increasing, we see that the points along the itinerary $x_i:=a_i(0)$ are all distinct.
Recall that we have an affine map
$$ a_i^{-1}: B(x_i,r)\to B(0,r/d_i)$$
for each $i$.
Choose
$$0<\eta<\min\{r_1,(1-c^{-1})r\}.$$
Since the itinerary has infinitely many distinct points, we can choose
representatives $0<x_j<x_i<1$ such that
$$0<\Delta:=x_i-x_j<\frac{\lambda^2}{2-\lambda}\eta.$$
Here, we require neither $i<j$ nor $j<i$, and all the following affine
calculations use compatible local lifts to $\mathbb{R}$.
Setting
$$\epsilon:=a_i^{-1}(x_j+\eta)=\frac{\eta-\Delta}{d_i}<\eta<r_1,$$
we have local affine bijections
$$
(0,\epsilon)\xrightarrow{a_i}(x_i,x_j+\eta)
\xrightarrow{a_j^{-1}}
\left(\frac{\Delta}{d_j},\frac{\eta}{d_j}\right)\subseteq(0,r_1).
$$
Indeed, $[x_i,x_j+\eta]$ is contained in both $B(x_i,r)$ and $B(x_j,r)$,
which are affine neighborhoods for $a_i^{-1}$ and $a_j^{-1}$, respectively.
Let $\hat c:=a_j^{-1}a_i$.

We also note that $\hat c$ has infinite order.
Since each expansion step has slope at least $c>1$,
$$\frac{\max\{d_i,d_j\}}{\min\{d_i,d_j\}}\ge c.$$
For $u_*:=-\Delta/(d_i-d_j)$, our choice of $\eta$ gives
$$\max\{d_i,d_j\}|u_*|\le\frac{c\Delta}{c-1}<r.$$
Thus both $a_i$ and $a_j$ are affine near $u_*$, and
$$a_i(u_*)=x_i+d_i u_*=x_j+d_j u_*=a_j(u_*).$$
Hence $\hat c$ fixes $u_*$ and has slope $d_i/d_j\ne1$ there,
so it has infinite order.

Set $b:=f^N\hat c$, where $N$ is sufficiently large that
$$\beta:=\lambda^N\frac{d_i}{d_j}<1,
\qquad \lambda^N\frac{\eta}{d_j}<\epsilon.$$
Since $(0,r_1)$ is an attracting affine neighborhood on the positive side of $0$ for $f$,
we obtain the affine bijection
$$b\restriction:(0,\epsilon)\longrightarrow
\left(\lambda^N\frac{\Delta}{d_j},
\lambda^N\frac{\eta}{d_j}\right)\subseteq(0,\epsilon).$$
Writing $C:=\Delta/d_i$, we have
$$b(x)=\beta(x+C)\qquad(0\le x\le\epsilon).$$
Finally,
$$\frac{C}{\epsilon}=\frac{\Delta}{\eta-\Delta}
<\frac{\lambda^2}{(2+\lambda)(1-\lambda)}.$$
The hypotheses of \cref{certi} are all verified.

As in \cref{certi}, we let 
$$s:=[f,b],w:=[s,fsf^{-1}]=\left[[f,b],[b,f^{-1}]\right].$$
By the same lemma, we have that 
$w$ and $sws^{-1}$ have a common fixed point, say $x_0$.
On the other hand, note that 
$$\form{H,f}\cong H\ast\form{f},$$
and that 
$\form{b,f}\cong F_2$.
It is an easy exercise in combinatorial group theory that 
in a rank-two free group $F(\alpha,\beta)$,
the elements $[\alpha,\beta]$ and $\alpha[\alpha,\beta]\alpha^{-1}$ generate a rank-two free subgroup. Applying this fact twice, we see that 
$w$ and $sws^{-1}$ generate a rank-two free group. 
Since those two elements fix $x_0$, we obtain a contradiction with \cref{thm:bs1985}.
\end{proof}

Let us now consider right-angled Artin subgroups of $\pls$.
The following is an elementary consequence of the normal form theorem for right-angled Artin groups; see~\cite{HM1995} for instance. 

\begin{lem}\label{lem:raag-power}
Let $\Gamma$ be a finite simplicial graph.
Then for every nonzero integer $N$, the subgroup
generated by 
$$\{v^N \;:\; v\in V(\Gamma)\}$$
in $A(\Gamma)$ is isomorphic to $A(\Gamma)$.
\end{lem}

The following is rather a soft obstruction for embeddings into $\pls$, compared to \cref{thm:free-prod}.

\begin{lem}\label{lem:f2-times-z}
The group $F_2\times\Z$ does not embed into $\pls$.
\end{lem}
\begin{proof}
Assume for contradiction that 
$$F_2\times\Z\cong \langle a,b\rangle\times\langle t\rangle$$ is realized as a subgroup of $\pls$.
Let us first consider the case that $t$ has a rational rotation number.
Using \cref{lem:raag-power}, we may pass to a power of $t$ and assume that $\Fix(t)$ is nonempty.
The group
$\langle a,b\rangle$ permutes the nonempty finite set $\partial\Fix(t)$; see Lemma~\ref{lem:boundary-finite}.
Passing to powers again if necessary, we may assume that $\langle a,b\rangle$ fixes $\partial\Fix(t)$.
We then obtain a faithful piecewise linear action of $\langle a,b\rangle\cong F_2$ on the complement of any one of these fixed points, contradicting \cref{thm:bs1985}.

Hence, the element $t$ has an irrational rotation number.
By Theorem~\ref{thm:herman}, the map $t$ is topologically conjugate to an irrational rotation, and hence, its centralizer group is abelian. This again contradicts the hypothesis that $\form{a,b}\cong F_2$.
\end{proof}

We can now characterize the right-angled Artin subgroups of $\pls$ as described in
\cref{cor:raag}, by combining Lemmas~\ref{lem:z2z} and \ref{lem:f2-times-z}.

\begin{proof}[Proof of Corollary~\ref{cor:raag}]
    By Lemma~\ref{lem:raag-power}, $A(\gam)$ embeds in $\mathrm{PL}(S^1)$ if and only if it embeds in $\pls $.
    
    If $\gam$ is a finite simplicial graph with at most two vertices then $A(\gam)$ is either free or abelian, in which case it embeds in $\pls $.
    
    Let $\gam$ be a finite simplicial graph that is not complete, has at least three vertices, and has at least one edge. Choose an arbitrary edge $e=\{v,w\}$ of $\gam$, subject to the condition that there is a vertex $z$ of $\gam$ that is nonadjacent to at least one of $v$ and $w$. Such a pair $(e,z)$ exists because $\gam$ is assumed not to be complete.

    If $z$ is adjacent to neither $v$ nor $w$ then the subgroup of $A(\gam)$ generated by $\{v,w,z\}$ is isomorphic to $\Z^2*\Z$, which cannot embed in $\pls $ by Lemma~\ref{lem:z2z}. If $z$ is adjacent to one of these vertices, then the subgroup of $A(\gam)$ generated by $\{v,w,z\}$ is isomorphic to $F_2\times\Z$, which cannot embed in $\pls $ by Lemma~\ref{lem:f2-times-z}.

    To establish the converse, it suffices to see that all finitely generated free and free abelian groups embed in $\pls $. For a free group, we already have that the free group $F_2$ of rank two embeds in $\pls $ since it already lies in Thompson's group $T$; this follows from a standard ping-pong argument; see~\cite{koberda-pingpong,dlHarpe2000}, for example. To show that all free abelian groups embed in $\pls $ is straightforward.
\end{proof}

\section{Recognition of virtually free groups in $\Homeo_+(S^1)$}\label{s:free}
In this section, we give the following criterion for virtually free subgroups of circle homeomorphism groups.
\begin{thm}\label{vfree-homeo}
If a finitely generated subgroup $G$ of $\Homeo_+(S^1)$ contains no finitely generated one-ended subgroup, then $G$ is virtually free.
\end{thm}

In \cite[Théorème 4]{Ghys1987Exceptional}, Ghys proved that if $G$ is a finitely generated subgroup of the orientation-preserving real-analytic diffeomorphisms on the circle $\Diff^\omega_+(S^1)$
with exceptional (i.e. Cantor) minimal set,
then $G$ is virtually free. His crucial topological tool was a \emph{surface cutting argument}, which can be abstracted as follows. For a surface $\Sigma$, a \emph{peripheral element} means an element of the surface group $\pi_1(\Sigma)$ represented by a boundary component.

\begin{lem}\label{lem:cutting}
Let $\Sigma$ be a compact connected oriented surface, and let $G$ be a group that contains no finitely generated one-ended subgroup. Suppose we have a representation $$\rho:\pi_1(\Sigma)\to G$$ such that the image of each boundary component has finite order.
Let $\mathcal{C}$ be a maximal collection of essential, non-peripheral, pairwise non-isotopic, and pairwise disjoint simple closed curves on $\Sigma$ such that $\rho(\gamma)$ has finite order for every $\gamma\in \CC$. Then $\rho(\pi_1(P))$ is finite for each component of $\Sigma\setminus\bigcup\CC$.
\end{lem}

Although Ghys considered only analytic diffeomorphisms in \cite{Ghys1987Exceptional}, we abstract his argument for the topological setting and give a complete proof here, using a group theoretic hypothesis of not containing a one-ended subgroup in the target group $G$.

\begin{proof}[Proof of \cref{lem:cutting}]
We first note that $ |\mathcal{C}|$ is bounded from above by the topological complexity of $\Sigma$. Hence, such a maximal (possibly empty) collection $\CC$ exists. 

Pick a component $P$ of $\Sigma\setminus\bigcup\CC$.
Assume for a contradiction that $L:=\rho(\pi_1(P))$ is infinite.
Let $\alpha$ denote the surjective map from $\pi_1(P)$ to $L$, induced by $\rho$.
By hypothesis, the group $L$ has more than one end, and so, $L$ splits over a finite subgroup $F$ as an amalgamated free product or an HNN extension (\cref{thm:ends}); here, we allow $L$ to have two ends, since every such group is virtually cyclic and since a virtually cyclic group is an HNN extension over a finite group or an amalgamated product of two finite groups over a finite subgroup of index two in both.  Also, every vertex group is a proper subgroup of $L$.
Let $Z$ be the corresponding graph of classifying spaces~\cite{Scott1979}, with the edge cylinder 
denoted as $Y\times[0,1]$ for $Y=K(F,1)$. 
We choose a continuous map $f:P\to Z$ realizing the surjection $\alpha$.

Every finite-order element of $L$ is conjugate into a vertex group, as can be easily seen from Bass--Serre theory. Note that $\alpha(c)$ has finite order for each $c\in \partial P$, since $c$ either comes from $\partial \Sigma$ or from $\CC$.
Hence, $f(c)$ is freely homotopic into vertex spaces for each component $c$ of $\partial P$. 
In particular, we may assume that $f(\partial P)$ is contained in the union of vertex spaces. 

Let $E:=Y\times\{1/2\}$. 
Using transversality, we may assume that $D:=f^{-1}(E)$ is a finite collection of simple closed curves; properly embedded arcs are not allowed since $f(\partial P)\cap E=\varnothing$. 
The idea of using transversality for maps from surfaces into graphs of spaces goes at least back to Culler~\cite{Culler1981}; see also \cite{CW2004, Kim2010} for more recent applications.

For every component $\gamma$ of $D$, the image $\alpha(\gamma)$ has finite order since $f(\gamma)\sse E$. 
By maximality of $\CC$, it follows that $\gamma$ is either null-homotopic or boundary-parallel. 
Choose a compact surface $P_0\sse P\setminus D$ such that the inclusion $P_0\to P$ induces a surjection $\pi_1(P_0)\to \pi_1(P)$; this is clear by removing disks and boundary-parallel annuli determined by $D$. 
Then the composition
$$P_0\xrightarrow{\mathrm{incl.}}P\xrightarrow{f} Z$$
induces a $\pi_1$-surjective map $P_0\to Z\setminus E$.
This is a contradiction since each vertex space of $Z$ cannot induce a $\pi_1$-surjection by inclusion. Hence, $L$ is finite.\end{proof}

\begin{proof}[Proof of \cref{vfree-homeo}]
We first establish that $[G,G]$ is torsion-free. For otherwise, we can find a nontrivial product of commutators
$$t=\prod_{i=1}^g[a_i,b_i]$$
with finite order. We then obtain a map from the fundamental group of a genus $g$-surface $\Sigma$ with one boundary into $G$.
Consider a collection of simple closed curves $\CC$ as in \cref{lem:cutting}.
For each connected component $P$ of $\Sigma\setminus\bigcup\CC$, we let $F_P$ denote the image of the group $\pi_1(P)$ under $\rho$. 
By \cref{lem:cutting}, the group $F_P$ is finite, so $\rot$ restricts
to a homomorphism on $F_P$.
Thus $\rot\circ\rho$ restricted to $\pi_1(P)$ factors through its abelianization.
Since the oriented boundary classes sum to zero in that abelianization, we have
$$\sum_{c\in\pi_0(\partial P)}\rot\circ\rho(c)=0.$$
We thus obtain
$$\rot(t)=\sum_{P\in\pi_0(\Sigma\setminus\bigcup\CC)}
\sum_{c\in\pi_0(\partial P)}\rot\circ\rho(c)=0.$$
Here, each interior cutting curve occurs twice with opposite orientations,
so these terms cancel, leaving only the original boundary component of $\Sigma$.
It follows that $\Fix t\ne\varnothing$, contradicting the fact that
$t$ is a nontrivial finite-order element.

We can write $G/[G,G]=\mathbb{Z}^r\oplus A$ for some finite abelian group $A$. Let $G_0$ be the kernel of the composition $$G\to G/[G,G] \to A.$$
Let $g\in G_0$ have finite order. If its image in $G/[G,G]$ is written as $g_1+g_2$ with $g_1\in\mathbb{Z}^r$ and $g_2\in A$, then $g_1=0$ because $g$ has finite order. Moreover, $g_2=0$ by definition. It follows that $G_0$ is a finitely generated, torsion-free finite-index subgroup of $G$. 

Let $H$ be a freely indecomposable free factor of $G_0$ in the Grushko decomposition of $G_0$. 
The group $H$ is not one-ended, by assumption.
It must be that $H$ has at least two ends. If $H$ has infinitely many ends then $H$ splits as a free product
over the trivial group, contrary to assumption. Thus, $H$ is cyclic and $G_0$ is free, so that $G$ is virtually free.
\end{proof}

\begin{rem}
The above proof actually shows that $G$ is free-by-finite-abelian. 
This point was also noted in \cite[Remarque 4.4]{Ghys1987Exceptional}.
\end{rem}

\section{Recognition of virtual surface groups}\label{s:surface}

In this section, we establish the following two results. These give the main theorem of this paper \emph{up to finite-index subgroups}. We will later upgrade the first theorem to the full generality of our main theorem.

\begin{thm}\label{thm:z2-free}
Every finitely generated, $\mathbb{Z}^2$--free subgroup of $\pls$ is either a virtually free group or a virtual surface group.
\end{thm}

\begin{thm}\label{thm:z2-1end}
Every finitely generated subgroup of $\PL_+(S^1)$ containing $\mathbb{Z}^2$ is non-acylindrically hyperbolic and one-ended.
\end{thm}

\subsection{Finite radicals and acylindrically hyperbolic groups}\label{ss:acyl}
Let $G$ be an abstract group. 
If there exists a maximal finite normal subgroup of $G$, then it is clear that there exists at most one such group; we call such a group the \emph{finite radical} of $G$, denoted as $K(G)$. 

\begin{rem}\label{rem:rad}
\begin{enumerate}
    \item Finite radicals generally do not exist, even for finitely presented groups~\cite{Abels1979}.
    \item If $K(G)$ exists, then $G/K(G)$ has no nontrivial finite normal subgroups.
    \item\label{p:rad}
    For a given group $G$, let us write
    $$\FF(G):=\{N\unlhd G : N\text{ is finite}\}.$$
    If $\sup_{N\in\FF(G)}|N|$ is bounded, 
    then a finite normal subgroup realizing that bound is the finite radical.
\end{enumerate}
\end{rem}

Recall that a geodesic path metric space $X$ is called \emph{Gromov hyperbolic} or just \emph{hyperbolic} if there is a $\delta>0$ for which $X$ satisfies the $\delta$--thin triangle condition. That is, for any geodesic triangle with vertices $\{a,b,c\}$ in $X$, any edge is contained in the $\delta$--neighborhood of the union of the other two edges. A group is \emph{word-hyperbolic} if the Cayley graph is Gromov hyperbolic, and is \emph{nonelementary} if it is infinite and does not contain a cyclic group of finite index.

Suppose $G$ is a group acting by isometries on a metric space $(X,d)$. The action of $G$ is \emph{acylindrical} if for all $r>0$ there exist $R,N>0$ such that whenever $x,y\in X$ with $d(x,y)\geq R$, then \[|\{g\in G\mid d(g x,x)\leq r,\, d(g y,y)\leq r\}|\leq N.\] A group is called \emph{acylindrically hyperbolic} if there is a Gromov hyperbolic metric space $X$ on which $G$ admits a \emph{nonelementary} 
(i.e.~with two loxodromics with disjoint endpoints) acylindrical action.

The following proposition summarizes standard facts about acylindrical hyperbolicity; see~\cite{DGO2017,osin-2016,MO2015}.

\begin{prop}\label{prop:acyl-back}
    Let $G$ be acylindrically hyperbolic.
    \begin{enumerate}
        \item Acylindrical hyperbolicity is inherited by finite-index subgroups.
        \end{enumerate}
Moreover, for a given acylindrical action on a hyperbolic metric space:
        \begin{enumerate}
        \setcounter{enumi}{1}
        \item Every nontrivial element $g\in G$ is either \emph{elliptic}, 
        i.e.~every orbit under the action of $g$ is bounded, or \emph{loxodromic}, i.e.~$g$ has a positive translation length.
        \item The centralizer of a loxodromic element is virtually cyclic.
    \end{enumerate}
\end{prop}

Suppose $G$ acts acylindrically on a Gromov hyperbolic metric space $X$. A subgroup of $G$ is called \emph{elliptic} if every element acts elliptically on $X$. With these assumptions, we have the following result of Abbott--Dahmani~\cite{abbott-dahmani}:

\begin{prop}\label{prop:abbott-dahmani}
    Suppose $G$ has no finite normal subgroups and acts coboundedly, nonelementarily, and acylindrically on a hyperbolic metric space $X$. If $\{H_1,\ldots,H_n\}$ are elliptic subgroups of $G$ then there is a nontrivial element $g\in G$ such that $\langle H_i,g\rangle\cong H_i*\langle g\rangle$ for all $1\leq i\leq n$.
\end{prop}

The element $g$ above can be chosen to be loxodromic; this is clear from the proof in \cite{abbott-dahmani}.
We also use the following well-known facts.
\begin{lem}[{\cite{DGO2017}}]\label{lem:acyl-rad}
If $G$ is acylindrically hyperbolic, then $K(G)$ exists and moreover, $G/K(G)$ is again acylindrically hyperbolic.
\end{lem}

The proof of the above lemma proceeds by bounding the supremum in \cref{p:rad} by a geometric argument. 
We use the same strategy to establish an interesting parallel below for our PL setting.
\begin{lem}\label{lem:rad}
If $G\le\pls$ is nonabelian,
then $G$ admits a finite radical $K(G)$ that is cyclic and central, and $G/K(G)$ embeds into $\pls$. 
\end{lem}
\bpf
Let $\FF$ be as in \cref{rem:rad},
and let $F\in\FF$ be arbitrary. 
By \cref{lem:rot-fin}, the rotation number
$\rot$ restricts to an injective group homomorphism from the necessarily cyclic group $F$ into $S^1$. 
As a subgroup of the additive group $S^1$, the group $F$ is cyclic.
For all $g\in G$ and $f\in F$, we have
$$\rot(gfg^{-1})=\rot f,$$
which implies that $[g,f]=1$. 
In other words, $F$ is central. 
Since $G$ is nonabelian, we see from \cref{thm:hölder} that some $h\in G$ fixes a point. The group $F$ freely permutes the finite set $\partial \Fix h$ and hence, 
$|F|\le |\partial \Fix h|$. \cref{rem:rad} now implies that $K(G)$ exists.
We have just shown that $K(G)$ is cyclic and central. 

Let $m$ be the order of $K(G)$. If $m=1$, the last assertion is immediate,
so assume $m\ge2$.
Let $a$ be the generator of $K(G)$ with rotation number $1/m$.
Choose $x_0\in S^1$ and an orientation-preserving affine bijection
$$\phi_0:[x_0,ax_0]\longrightarrow[0,1/m],$$
where the source is the positively oriented orbit arc.
Define $\phi:S^1\to S^1$ by
$$\phi(a^i x)=\phi_0(x)+i/m\pmod{1}
\qquad(x\in[x_0,ax_0],\ 0\le i<m).$$
The definitions agree at the endpoints of the orbit arcs.
On the $i$th arc, $\phi$ is the composition of $a^{-i}$ with $\phi_0$,
followed by a translation; hence $\phi\in\PL_+(S^1)$.
By construction, $\phi a\phi^{-1}=\mathrm{Rot}(1/m)$.

Let $S_{1/m}:=\mathbb{R}/(m^{-1}\mathbb{Z})$, viewed as the quotient
of $S^1$ by this rotation group.
Since $K(G)$ is central, every element of $\phi G\phi^{-1}$ descends
to a PL homeomorphism of $S_{1/m}$.
The formula
$$\rho(gK(G))[t]=[\phi g\phi^{-1}(t)]\qquad(t\in S^1)$$
therefore defines a homomorphism
$$\rho:G/K(G)\longrightarrow\PL_+(S_{1/m}),$$
where brackets denote projection to $S_{1/m}$.
A homeomorphism inducing the identity on this quotient is a deck transformation,
so the kernel before passing to $G/K(G)$ is precisely $K(G)$.
Thus $\rho$ is injective, proving the conclusion.
\epf

\subsection{Virtual properties of groups}\label{ss:virtual}
We note the following abstract group theoretic facts.
\begin{lem}\label{lem:ext}
Let $E$ be a central extension of a finitely generated group $Q$ by a finite cyclic group $C$. 
Let (P) denote one of the following properties for a group.
\begin{enumerate}
    \item The group contains $\mathbb{Z}^2$.
    \item The group is virtually free.
    \item The group is a virtual surface group.
\end{enumerate}
Then $E$ has the property (P) if and only if so does $Q$.
\end{lem}
\bpf
The forward implication is trivial, so we consider the reverse direction only. We let $p:E\to Q$ be the given projection.

For part (1), suppose $Q$ contains $\bar A=\form{\bar a,\bar b}\cong \mathbb{Z}^2$.
The preimages $a,b$ satisfy that $[a,b]\in C$. 
Since $C$ is finite central, we have some $k>0$ such that
$$[a^k,b]=[a,b]^k=1.$$
It follows that $\form{a^k,b}$ is a copy of $\mathbb{Z}^2$ in $E$.

Part (2) is straightforward from the universal property of a free group.
For part (3), it suffices to consider the case $Q=\pi_1(S)$,
where $S=S_g$ is a closed orientable surface of genus $g\ge2$.
The central extension
$$1\to C\to E\xrightarrow{p}\pi_1(S)\to1$$
corresponds to a cohomology class
$$\alpha\in H^2(\pi_1(S);C)=H^2(S;C).$$
Choose a connected degree-$|C|$ covering map $q:S'\to S$.
Then
$$\form{q^*\alpha,[S']}=\form{\alpha,q_*[S']}
=\form{\alpha,|C|\cdot[S]}=0.$$
Evaluation on $[S']$ identifies $H^2(S';C)$ with $C$, so $q^*\alpha=0$.
Consequently, the restricted central extension splits, giving
$$p^{-1}\bigl(q_*\pi_1(S')\bigr)\cong C\times\pi_1(S').$$
This subgroup has finite index in $E$ and itself contains $\pi_1(S')$
as a finite-index subgroup.
Thus $E$ is a virtual surface group.
\epf

We note the following standard facts on graphs of groups; see~\cite{Serre1977, Scott1979, Cohen1989} for background.

\begin{lem}\label{lem:tree}
The following conclusions hold.
\begin{enumerate}
    \item (\cite[Lemma 8.32]{Cohen1989}) Let $\mathcal{G}$ be a finite graph of groups with finitely generated edge groups. If $\pi_1(\mathcal{G})$ is finitely generated, then so is each vertex group of $\mathcal{G}$.
    \item (\cite{KPS1973}) A group $G$ is virtually free if and only if $G=\pi_1(\mathcal{G})$ for some finite graph $\mathcal{G}$ of finite vertex groups.
    \item (\cite{MO2015}; see \cite[Theorem 2.17]{Cohen2026} for this form) An action $\alpha$ of a group $G$ on a simplicial tree is acylindrical if and only if there exists some $k,B>0$ such that for every path $\gamma$ with length at least $k$, the pointwise stabilizer group of $\gamma$ has cardinality at most $B$.
    \item (\cite[Corollaries 2.2 and 2.3]{MO2015}) If $G$ is infinite-ended, then the Bass--Serre tree action corresponding to a splitting of the form in \cref{thm:ends} is acylindrical.
\end{enumerate}
\end{lem}

By an \emph{extended Fuchsian group}, we mean a discrete subgroup of $\Isom(\mathbb{H}^2)$, allowing orientation-reversing isometries.
Its \emph{peripheral groups} are defined as the cyclic groups which are a parabolic subgroup or the stabilizer of a component in the domain of the discontinuity. If $H$ is a finite-by-extended Fuchsian group with quotient $Q$, the lifts of the peripheral subgroups of $Q$ are called the peripheral subgroups of $H$; also, the group $H$ is a virtually free group or a virtual surface group, as follows from \cref{lem:ext}.
We recall the following consequence of the powerful Algebraic Torus Theorem of Dunwoody and Swenson.
\begin{thm}[{Algebraic Annulus Theorem~\cite{DS2000}}]\label{thm:annulus}
If $G$ is a finitely generated group, and if $J$ is a virtually cyclic subgroup such that $e(G,J)>1$, then one of the following holds.
\begin{enumerate}[(A)]
    \item\label{p:annulus1} $G$ splits over some group that is commensurable with $J$.
    \item\label{p:annulus2} $G$ acts properly discontinuously and cocompactly by (possibly orientation-reversing) isometries on the Euclidean plane $\mathbb{E}^2$ or on the hyperbolic plane $\mathbb{H}^2$.
    \item\label{p:annulus3} There is a finite graph of groups decomposition for $G$ with a distinguished vertex group $H$ containing $J$
    such that $H$ is a finite-by-extended Fuchsian group,
    and such that the edge groups correspond to the conjugates of the peripheral subgroups of $H$.
\end{enumerate}
\end{thm}

Using Theorem~\ref{thm:annulus}, we obtain a splitting theorem for one-ended subgroups in $\pls$ as follows.

\begin{lem}\label{lem:split}
Suppose $G$ is a finitely generated, one-ended, $\mathbb{Z}^2$--free subgroup of $\pls$ with trivial finite radical.
\begin{enumerate}
    \item The group $G$ is a virtual surface group, or $G$ splits over an infinite virtually cyclic group.
    \item If $G$ splits over an infinite virtually cyclic group $E$,
    then the corresponding action $G\curvearrowright T$ on the simplicial minimal Bass-Serre tree $T$ is faithful and acylindrical.
    \end{enumerate}
\end{lem}

\bpf
(1) 
By \cref{lem:cyc-stab}, we have some infinite cyclic group $J$ such that $e(G,J)>1$. By the Algebraic Annulus Theorem (Theorem~\ref{thm:annulus}) above, one of the three alternatives in the theorem holds. If the first alternative \eqref{p:annulus1} holds, then $G$ satisfies our conclusion.
Suppose the second alternative \eqref{p:annulus2} holds.
Since $G$ has no finite radical, the proper action of $G$ is faithful. The target space cannot be the Euclidean plane since $G$ is $\mathbb{Z}^2$--free. It follows that $G$ is an extended cocompact Fuchsian group, which is a virtual surface group.

Suppose now \eqref{p:annulus3} holds, so that we have such a distinguished vertex group $H$ with
$$1\rightarrow C\rightarrow H\rightarrow Q\rightarrow 1$$
for some extended Fuchsian group $Q$.
If the splitting is trivial, then we have $G=H$, which is a virtually free group or a virtual surface group by \cref{lem:ext}. 
Suppose the splitting is nontrivial.
All the edge groups are infinite since $G$ is one-ended~\cite{Stallings1971}.
The group $Q$ cannot be cocompact; for otherwise, $Q$ has no peripheral groups.
Hence $Q$ is non-cocompact, and there exists an edge group in the splitting given by the alternative that is an infinite virtually cyclic group, establishing our conclusion.

(2) Consider the action on the Bass--Serre tree $T$; each edge stabilizer group is conjugate to $E$. The faithfulness is a consequence of $K(G)=1$. 
Choose an infinite cyclic, finite-index subgroup $C$ of $E$. By \cref{lem:cyc-comm}, we have that
$$ M:=\Comm_G(E)= \Comm_G(C)$$
is infinite virtually cyclic, which is a finite extension of $E$.
As an elliptic subgroup, the group $E$ has a finite orbit.
Hence, $M$ also has a finite orbit in $T$, implying that $M$ fixes a point. 
In~\cite[Proposition 6.7]{GL2011GT}, it was shown that 
if a group $G$ acts on a simplicial tree $T$ in such a way that 
$\Comm_G(\stab_G(e))$ fixes a point of $T$ for every edge $e$,
and if every edge stabilizer group is two-ended,
then the action is \emph{almost acylindrical};
this means that there exists some $k>0$ such that
for all segments $\gamma$ in $T$ with length at least $k$,
the pointwise stabilizer $\stab_G(\gamma)$ of $\gamma$ is a finite group.
In our case, $\stab_G(\gamma)$ is conjugate into $E$.
Also, for every finite subgroup $F\le E$, we have
$F\cap C=\{1\}$, implying that
$$|F|\le [E:C].$$
It follows that 
$\stab_G(\gamma)$ has a universal bound $[E:C]$. 
The action is obviously cobounded (since it is a minimal tree action) 
and unbounded (since the splitting is nontrivial).
Since $G$ is not virtually cyclic, this action is nonelementary.
This verifies the acylindricity condition of the action, as given in \cref{lem:tree} (3); see also~\cite{MO2015}.
\epf

We note the following:

\begin{lem}\label{lem:elliptic}
Let $L\le\pls$ be a finitely generated group with trivial finite radical. 
Suppose $L$ admits a non-elementary, cobounded, acylindrical action on a hyperbolic space $X$.
Then every finitely generated elliptic subgroup $V$ of $L$ is virtually free.
\end{lem}
\bpf
Assume for contradiction that $V$ is not virtually free. By \cref{vfree-homeo}, we have some finitely generated one-ended subgroup $H$ in $V$. By \cref{prop:abbott-dahmani}, we have an embedding of $H\ast\mathbb{Z}$ in $L$. This violates \cref{thm:free-prod}.
\epf

We can now complete the recognition of virtual surface groups.
\begin{proof}[Proof of \cref{thm:z2-free}]
We may assume $G$ is nonabelian and not virtually cyclic.
We can further assume $G$ has trivial finite radical, by \cref{lem:rad}. 
Suppose $G$ is infinite-ended. In this case, we repeatedly use \cref{lem:tree}.
\cref{thm:ends} and \cref{lem:tree} imply that $G$ admits a nonelementary cobounded acylindrical action on a Bass--Serre tree $T$ with finite edge stabilizer groups.
Furthermore, each vertex stabilizer group is finitely generated; \cref{lem:elliptic} further implies that each vertex group is virtually free. This implies that the whole group $G$ is virtually free.

Assume now that $G$ is one-ended. By Lemma~\ref{lem:split} (1), the group $G$ is either a virtual surface group or $G$ splits over an infinite virtually cyclic group. In the latter case, $G$ admits an acylindrical action on the corresponding Bass-Serre tree $T$. 
Let $H$ be the stabilizer group of a vertex. 
By \cref{lem:elliptic} again, the group $H$ is virtually free.
If $G$ contains a copy of the Baumslag--Solitar group $\operatorname{BS}(m,n)$, then by \cref{lem:balance}, we have $|m|=|n|$, and hence, $G$ would contain $\mathbb{Z}^2$. It follows that $G$ is a graph of virtually free groups with virtually cyclic edge groups, and $G$ does not contain a Baumslag--Solitar subgroup. It follows from the Bestvina--Feighn combination theorem~\cite{BF1992} that $G$ is word-hyperbolic. 
A theorem due to Wilton~\cite[Theorem 6.1]{Wilton2018Essential} asserts that every one-ended word-hyperbolic graph of virtually free groups with virtually cyclic edge groups contains a quasiconvex surface group. Let $S$ be such a surface group in $G$.
One can show by an elementary ping-pong argument that whenever $A$ is an infinite-index quasiconvex subgroup of a non-elementary word-hyperbolic group with trivial radical, this latter group contains a copy of $A\ast\mathbb{Z}$; see~\cite[Theorem 1]{Arzhantseva2001} for a torsion-free case, and see \cite[Proposition 6.11]{AG2025} for the general case. We see from \cref{thm:free-prod} that $S$ cannot be of infinite index. It follows that $G$ contains $S$ as a finite-index subgroup, as desired.
\end{proof}

We now complete the non-acylindrical hyperbolicity characterization.
\begin{proof}[Proof of \cref{thm:z2-1end}]
By \cref{lem:acyl-rad}, \cref{lem:rad} and \cref{lem:ext}, we may assume that $G$ has trivial finite radical.
Fix a copy $\Z^2\cong A\le G$.

We claim that $G$ is non-acylindrically hyperbolic.
Assume not, for a contradiction, so that the group $G$ has a non-elementary acylindrical action on a hyperbolic space $X$.
The subgroup $A$ cannot contain any loxodromics, since the centralizer of a loxodromic element is virtually cyclic. In other words, $A$ is elliptic. By \cref{prop:abbott-dahmani}, we obtain a copy of $A\ast \mathbb{Z}$ inside of $G$, contradicting \cref{thm:free-prod}.

By \cref{lem:tree} (4), the above claim also implies that $G$ is not infinite-ended. Hence, $G$ is one-ended.    
\end{proof}

Non-existence of higher hyperbolic manifold groups is now a simple consequence.
\begin{proof}[{Proof of Corollary~\ref{cor:hyp-manif}}]
    Let $G=\pi_1(M)$ be as in the statement of the corollary, with $M$ a finite volume hyperbolic manifold. If $M$ is closed then $G$ is hyperbolic, torsion-free, and not virtually a surface group. Thus Theorem~\ref{thm:main} gives the desired conclusion.

    If $M$ is not closed then $M$ has at least one cusp. The cusp group contains a copy of $\Z^n$ for some $n\geq 2$. A standard ping-pong argument (for instance, using a criterion of Mart\'inez-Pedroza,~\cite{MP2009}) shows that then $G$ contains a copy of $\Z*\Z^2$, violating \cref{thm:free-prod}.
\end{proof}

\section{Boundary circle and orientability}\label{s:boundary}
In this section, we establish the main theorem of the paper. The only remaining ingredient for the theorem is to establish alternative~\eqref{p:torus} of \cref{thm:main} when $G\le\PL_+(S^1)$ is virtually a closed hyperbolic surface group. Essentially, it only remains for us to show the following.

\begin{prop}\label{prop:orientable}
Suppose $G\le\PL_+(S^1)$ is virtually a closed hyperbolic surface group. Then for some integer $m\ge1$ and for some compact orientable hyperbolic 2-orbifold $M$, we have a central extension
$$
1\longrightarrow \mathbb{Z}/m\longrightarrow G\longrightarrow \pi_1^{\mathrm{orb}}(M)\longrightarrow1.
$$
\end{prop}

\begin{rem}\label{rem:orientable}
\begin{enumerate}
\item
The $G\le\PL_+(S^1)$ hypothesis cannot be weakened to $G\le\Homeo_+(S^1)$. Indeed, if we denote by $N_{-1}$ the nonorientable closed hyperbolic surface with Euler characteristic $-1$, then $G:=\pi_1(N_{-1})$ is left-orderable (even bi-orderable~\cite{RolfsenWiest2001}) and hence $$G\hookrightarrow \Homeo_+(\mathbb{R})\hookrightarrow \Homeo_+(S^1).$$ However, the torsion-freeness of $G$ prohibits $G$ from fitting into a central extension of the above type.
\item One cannot strengthen the conclusion to take $M$ to be a surface. Abrams, Katok and Ugarcovici~\cite{AKU2024} constructed an embedding of $W_1:=\pi_1(S_g)\rtimes\mathbb{Z}/2$ into $\pls$, where the action of $\mathbb{Z}/2$ is conjugate to a half-turn. This group $W_1$ cannot satisfy the above conclusion with $M$ being a surface group.
\item The same paper~\cite{AKU2024} describes an embedding of $W_2:=\pi_1(S_g)\rtimes\mathbb{Z}/2$ into $\operatorname{PL}(S^1)$ with $\mathbb{Z}/2$ being conjugate to a reflection. The boundary action of $W_2$ contains an orientation-reversing element. In particular, $W_2$ is not isomorphic to a cocompact Fuchsian (orientable, by definition) group. It follows that the hypothesis of Proposition~\ref{prop:orientable} cannot be weakened to $G\le\PL(S^1)$.
\end{enumerate}
\end{rem}

\begin{proof}[Proof of \cref{prop:orientable}]
We may assume $G$ is normalized (\cref{lem:normalized}).
Let $S$ be a closed orientable hyperbolic surface whose $\pi_1$ embeds into $G$ as a finite-index subgroup.

The action of a word-hyperbolic group $G$ on the Gromov boundary $\partial G\approx S^1$ is a uniform convergence action~\cite{BowditchConfiguration1999}, and hence, the kernel $C$ is finite. Let $H\unlhd G$ be an arbitrary finite normal subgroup. For each $h\in H$ and for each geodesic sequence $\{g_n\}$ in $G$, we see that 
$$d_G(g_n,hg_n)=\|g_n^{-1}hg_n\|_G$$
ranges over a finite set. This implies that $H$ acts trivially on $\partial G$, and that $C$ is the maximal finite normal subgroup of $G$. Bowditch's formulation~\cite{Bowditch1998} of the convergence group theorem, originally due to Tukia, Gabai and Casson--Jungreis~\cite{Tukia1988, Gabai1992, CJ1994}, implies that $G/C\cong\pi_1^{orb}(\mathcal{O})$ for some (possibly nonorientable) compact hyperbolic 2-orbifold.

For now, assume \cref{lem:ori} below. 
Set $k:=\deg \Omega\ne0$. For each $g\in G\le\PL_+(S^1)$, we have that $\deg g=1$ and that 
$$k=\deg(\Omega\circ g)=\deg(\partial g\circ\Omega)=\deg\partial g\cdot k,
$$
which implies $\partial g:\partial G\to\partial G$ has degree 1. Identifying $\partial G$ with $\partial\mathbb{H}^2$ canonically, we obtain that $\mathcal{O}$ is orientable. The proposition is proved.
\end{proof}

The following lemma will now complete the proof of Proposition~\ref{prop:orientable}.

\begin{lem}\label{lem:ori}
Under the hypotheses of \cref{prop:orientable}, there exists a continuous $G$--equivariant map 
$$ \Omega: S^1\to \partial G$$
of non-zero degree.
\end{lem}

\begin{proof}
We still assume $G$ is normalized.
Applying \cref{expansion} to $G$, choose an expansion parameter $(E,c,C,r)$.
Fix a finite symmetric generating set $T$ of $G$ containing $E$.
Increasing $C$ and then decreasing $r$ if necessary, we may assume that
$C\ge\max_{t\in T}\|Dt\|_\infty$ while still having $0<r<1/(8C)$.
We use the left-invariant word metric
$$d_T(u,v):=\|u^{-1}v\|_T.$$
For each $x\in S^1$, choose a weak expansion sequence
$a_n=e_n\cdots e_1$ from $x$, with $a_0=1$.
We claim that, for all $0\le m<n$,
\begin{equation}\label{eq:ccnm}
\frac{\log c}{\log C}(n-m)\le\|a_na_m^{-1}\|_T\le n-m.
\end{equation}
The upper bound is trivial.
For the lower bound, write $a_na_m^{-1}=t_1\cdots t_\ell$,
where $\ell=\|a_na_m^{-1}\|_T$ and $t_i\in T$.
The global slope bound gives
$$C^\ell\ge D(a_na_m^{-1})(a_mx)
=\frac{Da_n(x)}{Da_m(x)}\ge c^{n-m}.$$
This proves the claim.
Since $d_T(a_m^{-1},a_n^{-1})=\|a_na_m^{-1}\|_T$,
the sequence $\{a_n^{-1}\}$ determines a quasi-geodesic ray
with parameters independent of $x$ and of the chosen weak expansion sequence.
Let $\Omega(x)\in\partial G$ be its forward endpoint.

We record the affine neighborhoods used below.
Writing $\alpha_n:=Da_n(x)$, the inverse expansion steps give affine bijections
$$a_n^{-1}:B(a_nx,r)\longrightarrow B(x,r/\alpha_n).$$
Indeed, each inverse step contracts the image ball into the affine domain
of the preceding inverse step.
Thus $B(x,r/\alpha_n)$ is an affine neighborhood for $a_n$.

We claim that $\Omega(x)$ is independent of the chosen weak expansion sequence.
Let $\{b_n\}$ be another such sequence from $x$, and write $\beta_m:=Db_m(x)$.
For each $n\ge1$, choose the least $m=m(n)$ such that $\beta_m\ge\alpha_n$.
Since $\beta_m\le C\beta_{m-1}$, we have
$$1\le\beta_m/\alpha_n\le C.$$
The map $c_n:=b_ma_n^{-1}$ is therefore affine on $B(a_nx,r/C)$
with slope in $[1,C]$.
By \cref{bounded-affine}, these maps belong to a fixed finite set.
Consequently,
$$d_T(a_n^{-1},b_m^{-1})=\|b_ma_n^{-1}\|_T$$
is uniformly bounded.
Since $m(n)\to\infty$, the inverse rays have the same endpoint.
Thus $\Omega$ is well-defined using either weak or strong expansion sequences.

Let $g\in G$ and $x\in S^1$.
Choose weak expansion sequences $\{a_n\}$ from $x$ and $\{b_m\}$ from $gx$,
and set $\alpha_n:=Da_n(x)$ and $\beta_m:=Db_m(gx)$.
Choose $\epsilon>0$ such that $g$ is affine on the positively oriented
interval $[x,x+\epsilon]$, with slope $\sigma:=D^+g(x)>0$.
This choice is possible even when $x$ is a breakpoint of $g$.
For all sufficiently large $n$, choose the least $m=m(n)$ such that
$\sigma\beta_m\ge\alpha_n$.
Then
$$1\le\frac{\sigma\beta_m}{\alpha_n}\le C.$$
For such $n$, the interval $[a_nx,a_nx+r/C]$ is mapped by $a_n^{-1}$
into $[x,x+\epsilon]$, and its subsequent image under $g$ is contained in
$[gx,gx+r/\beta_m]$.
Thus $b_mga_n^{-1}$ is affine on an interval of length $r/C$
with slope in $[1,C]$.
By \cref{bounded-affine}, these elements belong to a fixed finite set, so
$$d_T(ga_n^{-1},b_m^{-1})=\|b_mga_n^{-1}\|_T$$
is uniformly bounded.
As $m(n)\to\infty$, we obtain $\Omega(gx)=g\Omega(x)$.

To prove continuity, fix a strong expansion sequence $\{a_n\}$ from $x$.
For each $N$, \cref{expansion} \eqref{p:expansion2} gives a neighborhood of $x$
such that every point $y$ in that neighborhood admits a weak expansion sequence
$\{b_n\}$ with $b_n=a_n$ for $n\le N$.
The corresponding inverse rays have uniform quasi-geodesic parameters
and share their first $N$ terms.
Their endpoints therefore converge to $\Omega(x)$ as $N\to\infty$,
which proves continuity.

Assume for a contradiction that $\deg\Omega=0$.
After identifying $\partial G$ with $\mathbb{R}/\mathbb{Z}$,
let $\pi:\mathbb{R}\to\partial G$ be the covering map.
There is a lift $\widetilde\Omega:S^1\to\mathbb{R}$,
whose image is a compact interval $[a,b]$.
For each $g\in G$, equivariance and connectedness of $S^1$ allow us to choose
a lift $\widetilde{\partial g}$ of the boundary homeomorphism $\partial g$ such that
$$\widetilde\Omega\circ g
=\widetilde{\partial g}\circ\widetilde\Omega.$$
Since $g$ maps $S^1$ onto itself, this lift preserves $[a,b]$ and hence
permutes its endpoints.
It follows that $G$ preserves the nonempty finite set
$\{\pi(a),\pi(b)\}\subseteq\partial G$.
This contradicts non-elementarity of the boundary action.
Therefore $\deg\Omega\ne0$, completing the proof.
\end{proof}

We now finally establish the main result of this paper.
\begin{proof}[Proof of Theorem~\ref{thm:main}]
Let $G\le\pls$ be a finitely generated group. If $G$ contains $\mathbb{Z}^2$, then alternative~\eqref{p:z2} holds by \cref{thm:z2-1end}. Assume $G$ is $\mathbb{Z}^2$--free. Combining \cref{thm:z2-free} and \cref{prop:orientable}, we obtain alternatives~\eqref{p:free} and~\eqref{p:torus}.
\end{proof}

\begin{proof}[Proof of Corollary~\ref{cor:mcg}]
    The mapping class group of a surface of complexity at least one is acylindrically hyperbolic via its action on the curve graph; see~\cite{Bowditch2008}. If the complexity of the underlying surface is at least two then two Dehn twists about disjoint curves generate a copy of $\Z^2$. Thus, all finite-index subgroups contain copies of $\Z^2$, and so Theorem~\ref{thm:z2-1end} precludes a faithful action of a finite-index subgroup of the mapping class group by piecewise linear homeomorphisms.
\end{proof}

\bibliographystyle{plain}
\bibliography{ref}

@unpublished{MBT2026,
	author = {Matte Bon, Nicolás and Triestino, Michele},
	note = {preprint},
	title = {$\mathbb{Z}^2$-free subgroups of Thompson’s $T$ are virtually free},
    year = {2026},
    eprint        = {2608.16530},
archivePrefix = {arXiv},
primaryClass  = {math.GR},
url           = {https://arxiv.org/abs/2608.16530},
    }

@article {Ghys1987Exceptional,
    AUTHOR = {Ghys, \'Etienne},
     TITLE = {Classe d'{E}uler et minimal exceptionnel},
   JOURNAL = {Topology},
  FJOURNAL = {Topology. An International Journal of Mathematics},
    VOLUME = {26},
      YEAR = {1987},
    NUMBER = {1},
     PAGES = {93--105},
      ISSN = {0040-9383},
   MRCLASS = {57R30},
  MRNUMBER = {880511},
MRREVIEWER = {Lawrence\ Conlon},
       DOI = {10.1016/0040-9383(87)90024-3},
       URL = {https://doi.org/10.1016/0040-9383(87)90024-3},
}

@article{Cornulier2019,
title = {On the space of ends of infinitely generated groups},
journal = {Topology and its Applications},
volume = {263},
pages = {279-298},
year = {2019},
issn = {0166-8641},
doi = {https://doi.org/10.1016/j.topol.2019.05.013},
url = {https://www.sciencedirect.com/science/article/pii/S016686411930149X},
author = {Yves Cornulier},
}

@article {Freudenthal1931,
    AUTHOR = {Freudenthal, Hans},
     TITLE = {\"Uber die {E}nden topologischer {R}\"aume und {G}ruppen},
   JOURNAL = {Math. Z.},
  FJOURNAL = {Mathematische Zeitschrift},
    VOLUME = {33},
      YEAR = {1931},
    NUMBER = {1},
     PAGES = {692--713},
      ISSN = {0025-5874,1432-1823},
   MRCLASS = {99-04},
  MRNUMBER = {1545233},
       DOI = {10.1007/BF01174375},
       URL = {https://doi.org/10.1007/BF01174375},
}

@article {MP2009,
    AUTHOR = {Mart\'inez-Pedroza, Eduardo},
     TITLE = {Combination of quasiconvex subgroups of relatively hyperbolic
              groups},
   JOURNAL = {Groups Geom. Dyn.},
  FJOURNAL = {Groups, Geometry, and Dynamics},
    VOLUME = {3},
      YEAR = {2009},
    NUMBER = {2},
     PAGES = {317--342},
      ISSN = {1661-7207,1661-7215},
   MRCLASS = {20F67 (20F65)},
  MRNUMBER = {2486802},
MRREVIEWER = {Goulnara\ N.\ Arzhantseva},
       DOI = {10.4171/GGD/59},
       URL = {https://doi.org/10.4171/GGD/59},
}

@book {Stallings1971,
    AUTHOR = {Stallings, John},
     TITLE = {Group theory and three-dimensional manifolds},
    SERIES = {Yale Mathematical Monographs},
    VOLUME = {4},
      NOTE = {A James K. Whittemore Lecture in Mathematics given at Yale
              University, 1969},
 PUBLISHER = {Yale University Press, New Haven, Conn.-London},
      YEAR = {1971},
     PAGES = {v+65},
   MRCLASS = {57A10 (20E30)},
  MRNUMBER = {415622},
MRREVIEWER = {Edward\ M.\ Bloomberg},
}

@article {Triestino2019JTop,
    AUTHOR = {Alvarez, S\'ebastien and Filimonov, Dmitry and Kleptsyn,
              Victor and Malicet, Dominique and Meni\~no Cot\'on, Carlos and
              Navas, Andr\'es and Triestino, Michele},
     TITLE = {Groups with infinitely many ends acting analytically on the
              circle},
   JOURNAL = {J. Topol.},
  FJOURNAL = {Journal of Topology},
    VOLUME = {12},
      YEAR = {2019},
    NUMBER = {4},
     PAGES = {1315--1367},
      ISSN = {1753-8416,1753-8424},
   MRCLASS = {37C85 (20E06 37D25 37F15 57S05)},
  MRNUMBER = {3977878},
MRREVIEWER = {Shigenori\ Matsumoto},
       DOI = {10.1112/topo.12118},
       URL = {https://doi.org/10.1112/topo.12118},
}

@article {AAMMT2024,
    AUTHOR = {Alonso, Juan and Alvarez, S\'ebastien and Malicet, Dominique
              and Meni\~no Cot\'on, Carlos and Triestino, Michele},
     TITLE = {Ping-pong partitions and locally discrete groups of
              real-analytic circle diffeomorphisms, {I}: {C}onstruction},
   JOURNAL = {J. Comb. Algebra},
  FJOURNAL = {Journal of Combinatorial Algebra},
    VOLUME = {8},
      YEAR = {2024},
    NUMBER = {1-2},
     PAGES = {57--109},
      ISSN = {2415-6302,2415-6310},
   MRCLASS = {37C85 (20E06 20E08 57S05)},
  MRNUMBER = {4735233},
MRREVIEWER = {Anthony\ Genevois},
       DOI = {10.4171/jca/78},
       URL = {https://doi.org/10.4171/jca/78},
}

@article {BurslemWilkinson2004,
    AUTHOR = {Burslem, Lizzie and Wilkinson, Amie},
     TITLE = {Global rigidity of solvable group actions on {$S^1$}},
   JOURNAL = {Geom. Topol.},
  FJOURNAL = {Geometry and Topology},
    VOLUME = {8},
      YEAR = {2004},
     PAGES = {877--924},
      ISSN = {1465-3060,1364-0380},
   MRCLASS = {37C85 (22F05 37C15 37E10)},
  MRNUMBER = {2087072},
MRREVIEWER = {Michel\ Herr\'e\ Belliart},
       DOI = {10.2140/gt.2004.8.877},
       URL = {https://doi.org/10.2140/gt.2004.8.877},
}

@article {Shalen2001,
    AUTHOR = {Shalen, Peter B.},
     TITLE = {Three-manifolds and {B}aumslag-{S}olitar groups},
      NOTE = {Geometric topology and geometric group theory (Milwaukee, WI,
              1997)},
   JOURNAL = {Topology Appl.},
  FJOURNAL = {Topology and its Applications},
    VOLUME = {110},
      YEAR = {2001},
    NUMBER = {1},
     PAGES = {113--118},
      ISSN = {0166-8641,1879-3207},
   MRCLASS = {57M05 (20F05 57N10)},
  MRNUMBER = {1804702},
MRREVIEWER = {Stephen\ J.\ Pride},
       DOI = {10.1016/S0166-8641(99)00167-4},
       URL = {https://doi.org/10.1016/S0166-8641(99)00167-4},
}

@article {Zassenhaus1938,
    AUTHOR = {Zassenhaus, Hans},
     TITLE = {Beweis eines satzes \"uber diskrete gruppen},
   JOURNAL = {Abh. Math. Sem. Univ. Hamburg},
  FJOURNAL = {Abhandlungen aus dem Mathematischen Seminar der Universit\"at
              Hamburg},
    VOLUME = {12},
      YEAR = {1937},
    NUMBER = {1},
     PAGES = {289--312},
      ISSN = {0025-5858,1865-8784},
   MRCLASS = {99-04},
  MRNUMBER = {3069692},
       DOI = {10.1007/BF02948950},
       URL = {https://doi.org/10.1007/BF02948950},
}

@article {FK2012,
    AUTHOR = {Kleptsyn, V. A. and Filimonov, D. A.},
     TITLE = {The structure of groups of circle diffeomorphisms with the
              property of nonexpandable points being fixed},
   JOURNAL = {Funktsional. Anal. i Prilozhen.},
  FJOURNAL = {Funktsional\cprime ny\u i\ Analiz i ego Prilozheniya},
    VOLUME = {46},
      YEAR = {2012},
    NUMBER = {3},
     PAGES = {38--61},
      ISSN = {0374-1990,2305-2899},
   MRCLASS = {37E10 (37E30)},
  MRNUMBER = {3075039},
MRREVIEWER = {Igor\ N.\ Nikitin},
       DOI = {10.1007/s10688-012-0025-1},
       URL = {https://doi.org/10.1007/s10688-012-0025-1},
}

@article {CJ1994,
    AUTHOR = {Casson, Andrew and Jungreis, Douglas},
     TITLE = {Convergence groups and {S}eifert fibered {$3$}-manifolds},
   JOURNAL = {Invent. Math.},
  FJOURNAL = {Inventiones Mathematicae},
    VOLUME = {118},
      YEAR = {1994},
    NUMBER = {3},
     PAGES = {441--456},
      ISSN = {0020-9910,1432-1297},
   MRCLASS = {57M50 (57M60 57N10)},
  MRNUMBER = {1296353},
MRREVIEWER = {G.\ Peter\ Scott},
       DOI = {10.1007/BF01231540},
       URL = {https://doi.org/10.1007/BF01231540},
}

@incollection{BowditchConfiguration1999,
 author={Bowditch, Brian H.},
 title={Convergence groups and configuration spaces},
 booktitle={Geometric Group Theory Down Under},
 publisher={Walter de Gruyter},address={Berlin},pages={23--54},year={1999}
}

@article {Gabai1992,
    AUTHOR = {Gabai, David},
     TITLE = {Convergence groups are {F}uchsian groups},
   JOURNAL = {Ann. of Math. (2)},
  FJOURNAL = {Annals of Mathematics. Second Series},
    VOLUME = {136},
      YEAR = {1992},
    NUMBER = {3},
     PAGES = {447--510},
      ISSN = {0003-486X,1939-8980},
   MRCLASS = {20H10 (57M07)},
  MRNUMBER = {1189862},
MRREVIEWER = {Alan\ W.\ Reid},
       DOI = {10.2307/2946597},
       URL = {https://doi.org/10.2307/2946597},
}

@article{RolfsenWiest2001,
 author={Rolfsen, Dale and Wiest, Bert},
 title={Free group automorphisms, invariant orderings and topological applications},
 journal={Algebraic \& Geometric Topology},volume={1},pages={311--319},
 year={2001},doi={10.2140/agt.2001.1.311}
}

@article{Tukia1988,
  author = {Tukia, Pekka},
  title = {Homeomorphic conjugates of {F}uchsian groups},
  journal = {J. Reine Angew. Math.},
  volume = {391},
  pages = {1--54},
  year = {1988},
  doi = {10.1515/crll.1988.391.1}
}

@article{Asaoka2012,
  author  = {Asaoka, Masayuki},
  title   = {Rigidity of certain solvable actions on the sphere},
  journal = {Geometry \& Topology},
  volume  = {16},
  number  = {3},
  pages   = {1835--1857},
  year    = {2012},
  doi     = {10.2140/gt.2012.16.1835}
}

@misc{ButtonBalanced,
  author        = {Button, J. O.},
  title         = {Balanced groups and graphs of groups with infinite cyclic edge groups},
  year          = {2015},
  eprint        = {1509.05688},
  archivePrefix = {arXiv},
  primaryClass  = {math.GR},
  note          = {arXiv:1509.05688}
}

@article{GuelmanLiousse2011,
  author  = {Guelman, Nancy and Liousse, Isabelle},
  title   = {{$C^1$}-actions of {Baumslag--Solitar} groups on {$S^1$}},
  journal = {Algebraic \& Geometric Topology},
  volume  = {11},
  number  = {3},
  pages   = {1701--1707},
  year    = {2011},
  doi     = {10.2140/agt.2011.11.1701}
}

@book{Schmidt1977,
  author    = {Schmidt, Klaus},
  title     = {Cocycles on Ergodic Transformation Groups},
  series    = {Macmillan Lectures in Mathematics},
  volume    = {1},
  publisher = {Macmillan Co. of India},
  address   = {Delhi},
  year      = {1977}
}

@article {Houghton1972,
    AUTHOR = {Houghton, C. H.},
     TITLE = {Ends of groups and the associated first cohomology groups},
   JOURNAL = {J. London Math. Soc. (2)},
  FJOURNAL = {Journal of the London Mathematical Society. Second Series},
    VOLUME = {6},
      YEAR = {1972},
     PAGES = {81--92},
      ISSN = {0024-6107,1469-7750},
   MRCLASS = {20J05 (20F05)},
  MRNUMBER = {316595},
MRREVIEWER = {D.\ E.\ Cohen},
       DOI = {10.1112/jlms/s2-6.1.81},
       URL = {https://doi.org/10.1112/jlms/s2-6.1.81},
}

@article {Houghton1978,
    AUTHOR = {Houghton, C. H.},
     TITLE = {The first cohomology of a group with permutation module
              coefficients},
   JOURNAL = {Arch. Math. (Basel)},
  FJOURNAL = {Archiv der Mathematik},
    VOLUME = {31},
      YEAR = {1978/79},
    NUMBER = {3},
     PAGES = {254--258},
      ISSN = {0003-889X,1420-8938},
   MRCLASS = {20J10},
  MRNUMBER = {521478},
MRREVIEWER = {Kenneth\ S.\ Brown},
       DOI = {10.1007/BF01226445},
       URL = {https://doi.org/10.1007/BF01226445},
}

@article{LMBT2020,
    AUTHOR = {Lodha, Yash and Matte Bon, Nicol\'as and Triestino, Michele},
     TITLE = {Property {FW}, differentiable structures and smoothability of
              singular actions},
   JOURNAL = {J. Topol.},
  FJOURNAL = {Journal of Topology},
    VOLUME = {13},
      YEAR = {2020},
    NUMBER = {3},
     PAGES = {1119--1138},
      ISSN = {1753-8416,1753-8424},
   MRCLASS = {22D55 (37C85 37E10 57M60)},
  MRNUMBER = {4100128},
MRREVIEWER = {Dave\ Witte\ Morris},
       DOI = {10.1112/topo.12151},
       URL = {https://doi.org/10.1112/topo.12151},
}

@incollection{AKU2024,
 author={Abrams, Adam and Katok, Svetlana and Ugarcovici, Ilie},
 title={Rigidity of topological entropy of boundary maps associated to {Fuchsian} groups},
 booktitle={A Vision for Dynamics in the 21st Century},
 publisher={Cambridge University Press},year={2024},pages={19--47},
 doi={10.1017/9781009278898.002},
 note={}
}

@article {Arzhantseva2001,
    AUTHOR = {Arzhantseva, G. N.},
     TITLE = {On quasiconvex subgroups of word hyperbolic groups},
   JOURNAL = {Geom. Dedicata},
  FJOURNAL = {Geometriae Dedicata},
    VOLUME = {87},
      YEAR = {2001},
    NUMBER = {1-3},
     PAGES = {191--208},
      ISSN = {0046-5755,1572-9168},
   MRCLASS = {20F67 (57M07)},
  MRNUMBER = {1866849},
       DOI = {10.1023/A:1012040207144},
       URL = {https://doi.org/10.1023/A:1012040207144},
}

@incollection {Abels1979,
    AUTHOR = {Abels, Herbert},
     TITLE = {An example of a finitely presented solvable group},
 BOOKTITLE = {Homological group theory ({P}roc. {S}ympos., {D}urham, 1977)},
    SERIES = {London Math. Soc. Lecture Note Ser.},
    VOLUME = {36},
     PAGES = {205--211},
 PUBLISHER = {Cambridge Univ. Press, Cambridge-New York},
      YEAR = {1979},
      ISBN = {0-521-22729-1},
   MRCLASS = {20F05 (20E26 20F16)},
  MRNUMBER = {564423},
MRREVIEWER = {Ian\ M.\ Chiswell},
}

@article {Ghys1987GV,
    AUTHOR = {Ghys, \'{E}tienne},
     TITLE = {Sur l'invariance topologique de la classe de
              {G}odbillon-{V}ey},
   JOURNAL = {Ann. Inst. Fourier (Grenoble)},
  FJOURNAL = {Universit\'{e} de Grenoble. Annales de l'Institut Fourier},
    VOLUME = {37},
      YEAR = {1987},
    NUMBER = {4},
     PAGES = {59--76},
      ISSN = {0373-0956,1777-5310},
   MRCLASS = {57R30 (57R32 58F18)},
  MRNUMBER = {927391},
MRREVIEWER = {Steven\ E.\ Hurder},
       DOI = {10.5802/aif.1111},
       URL = {https://doi.org/10.5802/aif.1111},
}

@book{dMvS1993,
	address = {Berlin, Heidelberg},
	author = {de Melo, Welington and van Strien, Sebastian},
	doi = {10.1007/978-3-642-78043-1},
	isbn = {978-3-642-78045-5 978-3-642-78043-1},
	language = {en},
	publisher = {Springer Berlin Heidelberg},
	title = {One-{Dimensional} {Dynamics}},
	url = {http://link.springer.com/10.1007/978-3-642-78043-1},
	urldate = {2024-02-05},
	year = {1993}}

@article {BS2013,
    AUTHOR = {Bleak, Collin and Salazar-D\'iaz, Olga},
     TITLE = {Free products in {R}. {T}hompson's group {$V$}},
   JOURNAL = {Trans. Amer. Math. Soc.},
  FJOURNAL = {Transactions of the American Mathematical Society},
    VOLUME = {365},
      YEAR = {2013},
    NUMBER = {11},
     PAGES = {5967--5997},
      ISSN = {0002-9947,1088-6850},
   MRCLASS = {20F65 (20E06 37C85)},
  MRNUMBER = {3091272},
MRREVIEWER = {Andriy\ S.\ Ol\=\i\u inik},
       DOI = {10.1090/S0002-9947-2013-05823-0},
       URL = {https://doi.org/10.1090/S0002-9947-2013-05823-0},
}

@article{BS1985,
	author = {Brin, Matthew G. and Squier, Craig C.},
	doi = {10.1007/BF01388519},
	fjournal = {Inventiones Mathematicae},
	issn = {0020-9910,1432-1297},
	journal = {Invent. Math.},
	mrclass = {57S05 (20F32 57S25)},
	mrnumber = {782231},
	mrreviewer = {Bruno\ P.\ Zimmermann},
	number = {3},
	pages = {485--498},
	title = {Groups of piecewise linear homeomorphisms of the real line},
	url = {https://doi.org/10.1007/BF01388519},
	volume = {79},
	year = {1985}}

@article{KKR2021,
	author = {Kim, Sang-hyun and Koberda, Thomas and Rivas, Crist\'obal},
	doi = {10.3934/jmd.2021009},
	fjournal = {Journal of Modern Dynamics},
	issn = {1930-5311,1930-532X},
	journal = {J. Mod. Dyn.},
	mrclass = {57M60 (20F14 20F36 20F60 37C85 37E05)},
	mrnumber = {4288175},
	mrreviewer = {Andrzej\ W.\ Bi\'s},
	pages = {285--304},
	title = {Direct products, overlapping actions, and critical regularity},
	url = {https://doi.org/10.3934/jmd.2021009},
	volume = {17},
	year = {2021}}

@article {KKR2024,
    AUTHOR = {Kim, Sang-hyun and Koberda, Thomas and Rivas, Crist\'obal},
     TITLE = {Virtual critical regularity of mapping class group actions on
              the circle},
   JOURNAL = {Transform. Groups},
  FJOURNAL = {Transformation Groups},
    VOLUME = {29},
      YEAR = {2024},
    NUMBER = {3},
     PAGES = {1105--1114},
      ISSN = {1083-4362,1531-586X},
   MRCLASS = {57M60 (20F14 20F36 37C05 37C85 57K20)},
  MRNUMBER = {4788024},
MRREVIEWER = {Marja\ K.\ Kankaanrinta},
       DOI = {10.1007/s00031-022-09698-9},
       URL = {https://doi.org/10.1007/s00031-022-09698-9},
}

@article{KK2020crit,
	author = {{Kim}, {S.-h.} and Koberda, T.},
	doi = {10.1007/s00222-020-00953-y},
	fjournal = {Inventiones Mathematicae},
	issn = {0020-9910},
	journal = {Invent. Math.},
	mrclass = {57M60 (20F36 37C05 37C85 57S05)},
	mrnumber = {4121156},
	number = {2},
	pages = {421--501},
	title = {Diffeomorphism groups of critical regularity},
	url = {https://doi.org/10.1007/s00222-020-00953-y},
	volume = {221},
	year = {2020}}

@article{BKK2019JEMS,
	author = {Baik, H. and Kim, {S.-h.} and Koberda, T.},
	doi = {10.4171/jems/886},
	fjournal = {Journal of the European Mathematical Society (JEMS)},
	issn = {1435-9855},
	journal = {J. Eur. Math. Soc. (JEMS)},
	mrclass = {20F36 (20F65 37E10 57M60)},
	mrnumber = {4035847},
	number = {8},
	pages = {2333--2353},
	title = {Unsmoothable group actions on compact one-manifolds},
	url = {https://mathscinet.ams.org/mathscinet-getitem?mr=4035847},
	volume = {21},
	year = {2019}}

@article{KK2018JT,
	author = {Kim, {S.-h.} and Koberda, T.},
	doi = {10.1112/topo.12079},
	fjournal = {Journal of Topology},
	issn = {1753-8416},
	journal = {J. Topol.},
	mrclass = {57M60 (20F36 37C05 37C85 57S05)},
	mrnumber = {3989437},
	number = {4},
	pages = {1054--1076},
	title = {Free products and the algebraic structure of diffeomorphism groups},
	url = {https://mathscinet.ams.org/mathscinet-getitem?mr=3989437},
	volume = {11},
	year = {2018}}

@book{KK21-book,
	author = {Kim, Sang-hyun and Koberda, Thomas},
	note = {Springer},
	pages = {xiv+323},
	publisher = {Springer},
	series = {Springer Monographs in Mathematics},
	title = {Structure and regularity of group actions on one--manifolds},
	year = {2021}}

@article{LMT2019,
	adsurl = {https://ui.adsabs.harvard.edu/abs/2018arXiv180308567L},
	archiveprefix = {arXiv},
	author = {{Lodha}, Yash and {Matte Bon}, Nicol{\'a}s and {Triestino}, Michele},
	eid = {arXiv:1803.08567},
	eprint = {1803.08567},
	journal = {arXiv e-prints},
	month = {Mar},
	pages = {arXiv:1803.08567},
	primaryclass = {math.DS},
	title = {{Property FW, differentiable structures, and smoothability of singular actions}},
	year = {2018}}

@article {Kim:aa,
    AUTHOR = {Kim, Sang-Hyun and Koberda, Thomas},
     TITLE = {Integrability of moduli and regularity of {D}enjoy
              counterexamples},
   JOURNAL = {Discrete Contin. Dyn. Syst.},
  FJOURNAL = {Discrete and Continuous Dynamical Systems},
    VOLUME = {40},
      YEAR = {2020},
    NUMBER = {10},
     PAGES = {6061--6088},
      ISSN = {1078-0947,1553-5231},
   MRCLASS = {37E10 (37C05 37C15 37C85 37E45)},
  MRNUMBER = {4128339},
MRREVIEWER = {Puneet\ Sharma},
       DOI = {10.3934/dcds.2020259},
       URL = {https://doi.org/10.3934/dcds.2020259},
}

@article{MannWolff20,
	author = {Mann, K. and Wolff, M.},
	doi = {10.2140/gt.2020.24.1211},
	fjournal = {Geometry \& Topology},
	issn = {1465-3060},
	journal = {Geom. Topol.},
	mrclass = {57M60 (20F34 57K20 57M07)},
	mrnumber = {4157553},
	number = {3},
	pages = {1211--1223},
	title = {Rigidity of mapping class group actions on {$S^1$}},
	url = {https://doi.org/10.2140/gt.2020.24.1211},
	volume = {24},
	year = {2020}}

@book{dlHarpe2000,
	address = {Chicago, IL},
	author = {de la Harpe, Pierre},
	isbn = {0-226-31719-6; 0-226-31721-8},
	mrclass = {20F65 (20F69 57M07)},
	mrnumber = {1786869 (2001i:20081)},
	mrreviewer = {Lee Mosher},
	pages = {vi+310},
	publisher = {University of Chicago Press},
	series = {Chicago Lectures in Mathematics},
	title = {Topics in geometric group theory},
	year = {2000}}

@article{DKN2009,
	author = {Deroin, Bertrand and Kleptsyn, Victor and Navas, Andr\'{e}s},
	doi = {10.17323/1609-4514-2009-9-2-263-303},
	fjournal = {Moscow Mathematical Journal},
	issn = {1609-3321},
	journal = {Mosc. Math. J.},
	mrclass = {37C85 (37A50 37D25 37E10 37F15 60B15)},
	mrnumber = {2568439},
	mrreviewer = {Shigenori Matsumoto},
	number = {2},
	pages = {263--303, back matter},
	title = {On the question of ergodicity for minimal group actions on the circle},
	url = {https://doi.org/10.17323/1609-4514-2009-9-2-263-303},
	volume = {9},
	year = {2009}}

@book{KKM2019,
	author = {Kim, Sang-hyun and Koberda, Thomas and Mj, Mahan},
	doi = {10.1007/978-3-030-02855-8},
	isbn = {978-3-030-02854-1; 978-3-030-02855-8},
	mrclass = {37-02 (20F34 20F65 37C85 37D40 37F30 57M60)},
	mrnumber = {3887602},
	pages = {x+134},
	publisher = {Springer, Cham},
	series = {Lecture Notes in Mathematics},
	title = {Flexibility of group actions on the circle},
	url = {https://doi.org/10.1007/978-3-030-02855-8},
	volume = {2231},
	year = {2019}}

@article {Wilton2018Essential,
    AUTHOR = {Wilton, Henry},
     TITLE = {Essential surfaces in graph pairs},
   JOURNAL = {J. Amer. Math. Soc.},
  FJOURNAL = {Journal of the American Mathematical Society},
    VOLUME = {31},
      YEAR = {2018},
    NUMBER = {4},
     PAGES = {893--919},
      ISSN = {0894-0347,1088-6834},
   MRCLASS = {20F65 (20F67 57M07)},
  MRNUMBER = {3836561},
MRREVIEWER = {Vassilis\ Metaftsis},
       DOI = {10.1090/jams/901},
       URL = {https://doi.org/10.1090/jams/901},
}

@article {GL2011GT,
    AUTHOR = {Guirardel, Vincent and Levitt, Gilbert},
     TITLE = {Trees of cylinders and canonical splittings},
   JOURNAL = {Geom. Topol.},
  FJOURNAL = {Geometry \& Topology},
    VOLUME = {15},
      YEAR = {2011},
    NUMBER = {2},
     PAGES = {977--1012},
      ISSN = {1465-3060,1364-0380},
   MRCLASS = {20E08 (20F65)},
  MRNUMBER = {2821568},
MRREVIEWER = {Fran\c cois\ Dahmani},
       DOI = {10.2140/gt.2011.15.977},
       URL = {https://doi.org/10.2140/gt.2011.15.977},
}

@article{Denjoy1932,
author = {A. Denjoy},
journal = {Journal de Mathématiques Pures et Appliquées},
language = {fre},
pages = {333-376},
title = {Sur les courbes définies par les équations différentielles à la surface du tore },
url = {http://eudml.org/doc/234887},
volume = {11},
year = {1932},
}

@article{Holder1901,
	author = {H{\"o}lder, O.},
	journal = {Ber. Verh. S{\"a}chs. Akad. Wiss. Leipzig Math. Phys. Kl.},
	pages = {1--64},
	title = {Die Axiome der Quantit{\"a}t und die Lehre vom Ma{\ss}},
	username = {keinstein},
	verfuegbar = {SLUB: Acta Acad. 310-Math.},
	volume = 53,
	year = 1901}

@article{BKK2014,
	author = {Baik, H. and Kim, S. and Koberda, T.},
	doi = {10.1007/s11856-016-1307-8},
	fjournal = {Israel Journal of Mathematics},
	issn = {0021-2172},
	journal = {Israel J. Math.},
	mrclass = {20F36 (03D35 20E26 58D05)},
	mrnumber = {3509472},
	mrreviewer = {Valeriy G. Bardakov},
	number = {1},
	pages = {175--182},
	title = {Right-angled {A}rtin groups in the {$C^\infty$} diffeomorphism group of the real line},
	url = {http://dx.doi.org/10.1007/s11856-016-1307-8},
	volume = {213},
	year = {2016}}

@article{DGO2017,
	author = {Dahmani, F. and Guirardel, V. and Osin, D.},
	doi = {10.1090/memo/1156},
	fjournal = {Memoirs of the American Mathematical Society},
	isbn = {978-1-4704-2194-6; 978-1-4704-3601-8},
	issn = {0065-9266},
	journal = {Mem. Amer. Math. Soc.},
	mrclass = {20F65 (20F06 20F34 20F67 57M07)},
	mrnumber = {3589159},
	number = {1156},
	pages = {v+152},
	title = {Hyperbolically embedded subgroups and rotating families in groups acting on hyperbolic spaces},
	url = {http://dx.doi.org/10.1090/memo/1156},
	volume = {245},
	year = {2017}}

@incollection {koberda-survey,
    AUTHOR = {Koberda, Thomas},
     TITLE = {Geometry and combinatorics via right-angled {A}rtin groups},
 BOOKTITLE = {In the tradition of {T}hurston {II}. {G}eometry and groups},
     PAGES = {475--518},
 PUBLISHER = {Springer, Cham},
      YEAR = {[2022] \copyright 2022},
      ISBN = {978-3-030-97559-3; 978-3-030-97560-9},
   MRCLASS = {20F36 (05C45 05C48 05C50 05C60 20F65)},
  MRNUMBER = {4472060},
MRREVIEWER = {Valeriy\ G.\ Bardakov},
       DOI = {10.1007/978-3-030-97560-9\_15},
       URL = {https://doi.org/10.1007/978-3-030-97560-9_15},
}

@article{Herman1979,
	author = {M. Herman},
	journal = {Inst. Hautes {\'E}tudes Sci. Publ. Math. No. 49},
	pages = {5--233},
	title = {Sur la conjugaison differentiable des diffeomorphismes du cercle a des rotations},
	year = {1979}}

@article{athanassopoulos,
	author = {Athanassopoulos, K.},
	doi = {10.1016/j.exmath.2013.12.005},
	fjournal = {Expositiones Mathematicae},
	issn = {0723-0869},
	journal = {Expo. Math.},
	mrclass = {37E10 (37E45)},
	mrnumber = {3310927},
	mrreviewer = {Liviana Palmisano},
	number = {1},
	pages = {48--66},
	title = {Denjoy {$C^1$} diffeomorphisms of the circle and {M}c{D}uff's question},
	url = {https://doi.org/10.1016/j.exmath.2013.12.005},
	volume = {33},
	year = {2015}}

@book{Serre1977,
	author = {Serre, J.-P.},
	mrclass = {20H10 (22E40)},
	mrnumber = {0476875},
	mrreviewer = {E. Vinberg},
	note = {Avec un sommaire anglais, R\'edig\'e avec la collaboration de Hyman Bass, Ast\'erisque, No. 46},
	pages = {189 pp. (1 plate)},
	publisher = {Soci\'et\'e Math\'ematique de France, Paris},
	title = {Arbres, amalgames, {${\rm SL}_{2}$}},
	year = {1977}}

@incollection{Ghys1987,
	author = {Ghys, {\'E}.},
	booktitle = {The {L}efschetz centennial conference, {P}art {III} ({M}exico {C}ity, 1984)},
	mrclass = {58D05 (57S99 58F25)},
	mrnumber = {893858},
	mrreviewer = {John W. Wood},
	pages = {81--106},
	publisher = {Amer. Math. Soc., Providence, RI},
	series = {Contemp. Math.},
	title = {Groupes d'hom\'eomorphismes du cercle et cohomologie born\'ee},
	volume = {58},
	year = {1987}}

@incollection {koberda-pingpong,
    AUTHOR = {Koberda, Thomas},
     TITLE = {Ping-pong lemmas with applications to geometry and topology},
 BOOKTITLE = {Geometry, topology and dynamics of character varieties},
    SERIES = {Lect. Notes Ser. Inst. Math. Sci. Natl. Univ. Singap.},
    VOLUME = {23},
     PAGES = {139--158},
 PUBLISHER = {World Sci. Publ., Hackensack, NJ},
      YEAR = {2012},
      ISBN = {978-981-4401-35-7; 981-4401-35-8},
   MRCLASS = {57M60 (20E06 20F36)},
  MRNUMBER = {2987617},
MRREVIEWER = {Qiang\ Zhang},
       DOI = {10.1142/9789814401364\_0004},
       URL = {https://doi.org/10.1142/9789814401364_0004},
}

@article{KK2013b,
	author = {Kim, S. and Koberda, T.},
	doi = {10.1142/S021819671450009X},
	fjournal = {International Journal of Algebra and Computation},
	issn = {0218-1967},
	journal = {Internat. J. Algebra Comput.},
	mrclass = {20F36 (05E45)},
	mrnumber = {3192368},
	mrreviewer = {Thomas Haettel},
	number = {2},
	pages = {121--169},
	title = {The geometry of the curve graph of a right-angled {A}rtin group},
	url = {http://dx.doi.org/10.1142/S021819671450009X},
	volume = {24},
	year = {2014}}

@book{Navas2011,
	author = {Navas, A.},
	doi = {10.7208/chicago/9780226569505.001.0001},
	edition = {{S}panish},
	isbn = {978-0-226-56951-2; 0-226-56951-9},
	mrclass = {37E10 (37C05 37C15 37E45 57M60)},
	mrnumber = {2809110},
	pages = {xviii+290},
	publisher = {University of Chicago Press, Chicago, IL},
	series = {Chicago Lectures in Mathematics},
	title = {Groups of circle diffeomorphisms},
	url = {http://dx.doi.org/10.7208/chicago/9780226569505.001.0001},
	year = {2011}}

@article{FF2001,
	adsurl = {http://adsabs.harvard.edu/abs/2001math......7085F},
	author = {Farb, B. and Franks, J.},
	eprint = {math/0107085},
	journal = {ArXiv Mathematics e-prints},
	month = jul,
	title = {{Groups of homeomorphisms of one-manifolds, I: actions of nonlinear groups}},
	year = 2001}

@incollection{Kopell1970,
	author = {Kopell, N.},
	booktitle = {Global {A}nalysis ({P}roc. {S}ympos. {P}ure {M}ath., {V}ol. {XIV}, {B}erkeley, {C}alif., 1968)},
	mrclass = {57.47},
	mrnumber = {0270396 (42 \#5285)},
	mrreviewer = {W. Barth},
	pages = {165--184},
	publisher = {Amer. Math. Soc., Providence, R.I.},
	title = {Commuting diffeomorphisms},
	year = {1970}}

@article{FF2003,
	author = {Farb, B. and Franks, J.},
	doi = {10.1017/S0143385702001712},
	fjournal = {Ergodic Theory and Dynamical Systems},
	issn = {0143-3857},
	journal = {Ergodic Theory Dynam. Systems},
	mrclass = {58D05 (22E65 37B99 37E05 37E10 57M60)},
	mrnumber = {2018608 (2004k:58013)},
	mrreviewer = {Tomasz Rybicki},
	number = {5},
	pages = {1467--1484},
	title = {Groups of homeomorphisms of one-manifolds. {III}. {N}ilpotent subgroups},
	url = {http://dx.doi.org/10.1017/S0143385702001712},
	volume = {23},
	year = {2003}}

@article{Jorquera,
	author = {Jorquera, E.},
	coden = {TIAPD9},
	doi = {10.1016/j.topol.2012.02.003},
	fjournal = {Topology and its Applications},
	issn = {0166-8641},
	journal = {Topology Appl.},
	mrclass = {58D05 (37E05)},
	mrnumber = {2902746},
	mrreviewer = {Dave Witte Morris},
	number = {8},
	pages = {2115--2126},
	title = {A universal nilpotent group of {$C^1$} diffeomorphisms of the interval},
	url = {http://dx.doi.org/10.1016/j.topol.2012.02.003},
	volume = {159},
	year = {2012}}

@article{BF1992,
	author = {Bestvina, M. and Feighn, M.},
	coden = {JDGEAS},
	fjournal = {Journal of Differential Geometry},
	issn = {0022-040X},
	journal = {J. Differential Geom.},
	mrclass = {53C23 (57M50)},
	mrnumber = {1152226 (93d:53053)},
	mrreviewer = {Tadeusz Januszkiewicz},
	number = {1},
	pages = {85--101},
	title = {A combination theorem for negatively curved groups},
	url = {http://projecteuclid.org/euclid.jdg/1214447806},
	volume = {35},
	year = {1992}}

@article{Ghys2001,
	author = {Ghys, {\'E}.},
	coden = {ENMAAR},
	fjournal = {L'Enseignement Math\'ematique. Revue Internationale. IIe S\'erie},
	issn = {0013-8584},
	journal = {Enseign. Math. (2)},
	mrclass = {37C85 (37E10 57S05)},
	mrnumber = {1876932 (2003a:37032)},
	mrreviewer = {Grant Cairns},
	number = {3-4},
	pages = {329--407},
	title = {Groups acting on the circle},
	volume = {47},
	year = {2001}}

@article{GS1987,
	author = {Ghys, {\'E}. and Sergiescu, V.},
	coden = {COMHAX},
	doi = {10.1007/BF02564445},
	fjournal = {Commentarii Mathematici Helvetici},
	issn = {0010-2571},
	journal = {Comment. Math. Helv.},
	mrclass = {57S30 (57R30 57R32 57R50 58D05 58F18)},
	mrnumber = {896095 (90c:57035)},
	mrreviewer = {Takashi Tsuboi},
	number = {2},
	pages = {185--239},
	title = {Sur un groupe remarquable de diff\'eomorphismes du cercle},
	url = {http://dx.doi.org/10.1007/BF02564445},
	volume = {62},
	year = {1987}}

@article{Agol2013,
	author = {Agol, Ian},
	fjournal = {Documenta Mathematica},
	issn = {1431-0635},
	journal = {Doc. Math.},
	mrclass = {20F67 (57Mxx)},
	mrnumber = {3104553},
	note = {With an appendix by Agol, Daniel Groves, and Jason Manning},
	pages = {1045--1087},
	title = {The virtual {H}aken conjecture},
	volume = {18},
	year = {2013}}

@incollection {Matucci-2010,
    AUTHOR = {Matucci, Francesco},
     TITLE = {Mather invariants in groups of piecewise-linear
              homeomorphisms},
 BOOKTITLE = {Combinatorial and geometric group theory},
    SERIES = {Trends Math.},
     PAGES = {251--260},
 PUBLISHER = {Birkh\"auser/Springer Basel AG, Basel},
      YEAR = {2010},
      ISBN = {978-3-7643-9910-8},
   MRCLASS = {37E05 (20E45 37E10)},
  MRNUMBER = {2744023},
MRREVIEWER = {Nick\ Gill},
       DOI = {10.1007/978-3-7643-9911-5\_10},
       URL = {https://doi.org/10.1007/978-3-7643-9911-5_10},
}

@article{Bowditch1998,
	author = {Bowditch, Brian H.},
	coden = {ACMAA8},
	doi = {10.1007/BF02392898},
	fjournal = {Acta Mathematica},
	issn = {0001-5962},
	journal = {Acta Math.},
	mrclass = {20F32 (20E06 20E08 57M07)},
	mrreviewer = {Eric M. Freden},
	number = {2},
	pages = {145--186},
	title = {Cut points and canonical splittings of hyperbolic groups},
	url = {http://dx.doi.org/10.1007/BF02392898},
	volume = {180},
	year = {1998}}

@book {Cohen1989,
    AUTHOR = {Cohen, Daniel E.},
     TITLE = {Combinatorial group theory: a topological approach},
    SERIES = {London Mathematical Society Student Texts},
    VOLUME = {14},
 PUBLISHER = {Cambridge University Press, Cambridge},
      YEAR = {1989},
     PAGES = {x+310},
      ISBN = {0-521-34133-7; 0-521-34936-2},
   MRCLASS = {20-01 (20E06 20F05 57M05 57M10 57M20)},
  MRNUMBER = {1020297},
MRREVIEWER = {Stephen\ J.\ Pride},
       DOI = {10.1017/CBO9780511565878},
       URL = {https://doi.org/10.1017/CBO9780511565878},
}

@ARTICLE{koberda-girth-2026,
  title         = "Right-angled Artin groups of large girth and finite volume
                   hyperbolic $3$--manifold groups",
  author        = "Koberda, Thomas",
  month         =  jul,
  year          =  2026,
  copyright     = "http://arxiv.org/licenses/nonexclusive-distrib/1.0/",
  archivePrefix = "arXiv",
  primaryClass  = "math.GT",
  url = {http://arxiv.org/abs/2607.01652},
  eprint        = "2607.01652"
}

@incollection {Minakawa-2001,
    AUTHOR = {Minakawa, Hiroyuki},
     TITLE = {Deformations of {PL}-representations of surface groups},
 BOOKTITLE = {Essays on geometry and related topics, {V}ol. 1, 2},
    SERIES = {Monogr. Enseign. Math.},
    VOLUME = {38},
     PAGES = {557--575},
 PUBLISHER = {Enseignement Math., Geneva},
      YEAR = {2001},
      ISBN = {2-940264-05-8},
   MRCLASS = {57R30 (57R32 57R50 57S30 58D05 58H15)},
  MRNUMBER = {1929339},
MRREVIEWER = {Takashi\ Tsuboi},
}

@article{Wise2000,
	author = {Wise, D. T.},
	fjournal = {The Quarterly Journal of Mathematics},
	issn = {0033-5606},
	journal = {Q. J. Math.},
	mrclass = {20E26 (20F65 57M07)},
	mrreviewer = {Ilya Kapovich},
	number = {1},
	pages = {107--129},
	title = {Subgroup separability of graphs of free groups with cyclic edge groups},
	volume = {51},
	year = {2000}}

@article {osin-2016,
    AUTHOR = {Osin, D.},
     TITLE = {Acylindrically hyperbolic groups},
   JOURNAL = {Trans. Amer. Math. Soc.},
  FJOURNAL = {Transactions of the American Mathematical Society},
    VOLUME = {368},
      YEAR = {2016},
    NUMBER = {2},
     PAGES = {851--888},
      ISSN = {0002-9947,1088-6850},
   MRCLASS = {20F67 (20F65)},
  MRNUMBER = {3430352},
MRREVIEWER = {Alessandro\ Sisto},
       DOI = {10.1090/tran/6343},
       URL = {https://doi.org/10.1090/tran/6343},
}

@article {DS2000,
    AUTHOR = {Dunwoody, M. J. and Swenson, E. L.},
     TITLE = {The algebraic torus theorem},
   JOURNAL = {Invent. Math.},
  FJOURNAL = {Inventiones Mathematicae},
    VOLUME = {140},
      YEAR = {2000},
    NUMBER = {3},
     PAGES = {605--637},
      ISSN = {0020-9910,1432-1297},
   MRCLASS = {20F65},
  MRNUMBER = {1760752},
MRREVIEWER = {Ilya\ Kapovich},
       DOI = {10.1007/s002220000063},
       URL = {https://doi.org/10.1007/s002220000063},
}

@article {Cohen2026,
    AUTHOR = {Cohen, William Daniel},
     TITLE = {Groups acting acylindrically on trees},
   JOURNAL = {J. Group Theory},
  FJOURNAL = {Journal of Group Theory},
    VOLUME = {29},
      YEAR = {2026},
    NUMBER = {1},
     PAGES = {1--40},
      ISSN = {1433-5883,1435-4446},
   MRCLASS = {20E08 (20F67 57K32)},
  MRNUMBER = {5008338},
       DOI = {10.1515/jgth-2024-0133},
       URL = {https://doi.org/10.1515/jgth-2024-0133},
}

@article {KPS1973,
    AUTHOR = {Karrass, A. and Pietrowski, A. and Solitar, D.},
     TITLE = {Finite and infinite cyclic extensions of free groups},
   JOURNAL = {J. Austral. Math. Soc.},
  FJOURNAL = {J. Austral. Math. Soc.},
    VOLUME = {16},
      YEAR = {1973},
     PAGES = {458--466},
   MRCLASS = {20F05},
  MRNUMBER = {349850},
MRREVIEWER = {Stephen\ Meskin},
}

@article {MO2015,
    AUTHOR = {Minasyan, Ashot and Osin, Denis},
     TITLE = {Acylindrical hyperbolicity of groups acting on trees},
   JOURNAL = {Math. Ann.},
  FJOURNAL = {Mathematische Annalen},
    VOLUME = {362},
      YEAR = {2015},
    NUMBER = {3-4},
     PAGES = {1055--1105},
      ISSN = {0025-5831,1432-1807},
   MRCLASS = {20F67 (20E06 20E08 20F65 57M05)},
  MRNUMBER = {3368093},
MRREVIEWER = {Mohammad\ Shahryari},
       DOI = {10.1007/s00208-014-1138-z},
       URL = {https://doi.org/10.1007/s00208-014-1138-z},
}

@article {abbott-dahmani,
    AUTHOR = {Abbott, Carolyn R. and Dahmani, Fran\c cois},
     TITLE = {Property {$P_{naive}$} for acylindrically hyperbolic groups},
   JOURNAL = {Math. Z.},
  FJOURNAL = {Mathematische Zeitschrift},
    VOLUME = {291},
      YEAR = {2019},
    NUMBER = {1-2},
     PAGES = {555--568},
      ISSN = {0025-5874,1432-1823},
   MRCLASS = {20F65 (20E07 46L35)},
  MRNUMBER = {3936081},
MRREVIEWER = {Xianjin\ Wang},
       DOI = {10.1007/s00209-018-2094-1},
       URL = {https://doi.org/10.1007/s00209-018-2094-1},
}

@article {AG2025,
    AUTHOR = {Azuelos, P\'en\'elope and Gaboriau, Damien},
     TITLE = {Perfect kernel and dynamics: from {B}ass-{S}erre theory to
              hyperbolic groups},
   JOURNAL = {Math. Ann.},
  FJOURNAL = {Mathematische Annalen},
    VOLUME = {391},
      YEAR = {2025},
    NUMBER = {3},
     PAGES = {4733--4789},
      ISSN = {0025-5831,1432-1807},
   MRCLASS = {37C85 (20F65 20F67 37B20 43A07)},
  MRNUMBER = {4865255},
       DOI = {10.1007/s00208-024-03038-w},
       URL = {https://doi.org/10.1007/s00208-024-03038-w},
}

@incollection{Scott1979,
	address = {Cambridge},
	author = {Scott, Peter and Wall, Terry},
	booktitle = {Homological group theory ({P}roc. {S}ympos., {D}urham, 1977)},
	mrclass = {57M05 (20E06)},
	mrnumber = {564422 (81m:57002)},
	mrreviewer = {W. Jaco},
	pages = {137--203},
	publisher = {Cambridge Univ. Press},
	series = {London Math. Soc. Lecture Note Ser.},
	title = {Topological methods in group theory},
	volume = {36},
	year = {1979}}

@article{Charney2007,
	author = {Charney, Ruth},
	coden = {GEMDAT},
	doi = {10.1007/s10711-007-9148-6},
	fjournal = {Geometriae Dedicata},
	issn = {0046-5755},
	journal = {Geom. Dedicata},
	mrclass = {20F36 (20F65)},
	mrreviewer = {Noelle C. Antony},
	pages = {141--158},
	title = {An introduction to right-angled {A}rtin groups},
	url = {http://dx.doi.org/10.1007/s10711-007-9148-6},
	volume = {125},
	year = {2007}}

@article{Culler1981,
	author = {Culler, M.},
	coden = {TPLGAF},
	doi = {10.1016/0040-9383(81)90033-1},
	fjournal = {Topology. An International Journal of Mathematics},
	issn = {0040-9383},
	journal = {Topology},
	mrclass = {20E05 (57M99)},
	mrnumber = {605653 (82c:20052)},
	mrreviewer = {M. J. Dunwoody},
	number = {2},
	pages = {133--145},
	title = {Using surfaces to solve equations in free groups},
	url = {http://dx.doi.org/10.1016/0040-9383(81)90033-1},
	volume = {20},
	year = {1981}}

@article{CW2004,
	author = {Crisp, John and Wiest, Bert},
	doi = {10.2140/agt.2004.4.439},
	fjournal = {Algebraic \& Geometric Topology},
	issn = {1472-2747},
	journal = {Algebr. Geom. Topol.},
	mrclass = {20F36 (05C25 20F65 55R80 57M15 57M27)},
	mrnumber = {2077673 (2005e:20052)},
	mrreviewer = {Luis Paris},
	pages = {439--472},
	title = {Embeddings of graph braid and surface groups in right-angled {A}rtin groups and braid groups},
	url = {http://dx.doi.org/10.2140/agt.2004.4.439},
	volume = {4},
	year = {2004}}

@article{Agol2008,
	author = {Agol, Ian},
	doi = {10.1112/jtopol/jtn003},
	fjournal = {Journal of Topology},
	issn = {1753-8416},
	journal = {J. Topol.},
	mrclass = {57M50 (57N10)},
	mrnumber = {2399130 (2009b:57033)},
	mrreviewer = {Darren D. Long},
	number = {2},
	pages = {269--284},
	title = {Criteria for virtual fibering},
	url = {http://dx.doi.org/10.1112/jtopol/jtn003},
	volume = {1},
	year = {2008}}

@article{Bowditch2008,
	author = {Bowditch, Brian H.},
	coden = {INVMBH},
	doi = {10.1007/s00222-007-0081-y},
	fjournal = {Inventiones Mathematicae},
	issn = {0020-9910},
	journal = {Invent. Math.},
	mrclass = {57M50 (20F65)},
	mrnumber = {2367021 (2008m:57040)},
	mrreviewer = {Jason A. Behrstock},
	number = {2},
	pages = {281--300},
	title = {Tight geodesics in the curve complex},
	url = {http://dx.doi.org/10.1007/s00222-007-0081-y},
	volume = {171},
	year = {2008}}

@book{BH1999,
	address = {Berlin},
	author = {Bridson, M. R. and Haefliger, A.},
	isbn = {3-540-64324-9},
	mrclass = {53C23 (20F65 53C70 57M07)},
	mrreviewer = {Athanase Papadopoulos},
	pages = {xxii+643},
	publisher = {Springer-Verlag},
	series = {Grundlehren der Mathematischen Wissenschaften [Fundamental Principles of Mathematical Sciences]},
	title = {Metric spaces of non-positive curvature},
	volume = {319},
	year = {1999}}

@article{GS1999,
	author = {Guba, V. S. and Sapir, M. V.},
	doi = {http://dx.doi.org/10.1070/SM1999v190n08ABEH000419},
	fjournal = {Rossi\u\i skaya Akademiya Nauk. Matematicheski\u\i\ Sbornik},
	issn = {0368-8666},
	journal = {Mat. Sb.},
	mrclass = {20F05 (05C25)},
	mrnumber = {1725439 (2001m:20045)},
	mrreviewer = {G. A. Noskov},
	number = {8},
	pages = {3--60},
	title = {On subgroups of the {R}. {T}hompson group {$F$} and other diagram groups},
	url = {},
	volume = {190},
	year = {1999}}

@article{Kim2010,
	author = {Kim, {Sang-hyun}},
	doi = {10.4171/GGD/84},
	fjournal = {Groups, Geometry, and Dynamics},
	issn = {1661-7207},
	journal = {Groups Geom. Dyn.},
	mrclass = {20F36 (05C25 20F65 57M07)},
	mrnumber = {2595093 (2011d:20073)},
	mrreviewer = {Matthew B. Day},
	number = {2},
	pages = {275--307},
	title = {On right-angled {A}rtin groups without surface subgroups},
	url = {http://dx.doi.org/10.4171/GGD/84},
	volume = {4},
	year = {2010}}

@article {HM1995,
    AUTHOR = {Hermiller, Susan and Meier, John},
     TITLE = {Algorithms and geometry for graph products of groups},
   JOURNAL = {J. Algebra},
  FJOURNAL = {Journal of Algebra},
    VOLUME = {171},
      YEAR = {1995},
    NUMBER = {1},
     PAGES = {230--257},
      ISSN = {0021-8693,1090-266X},
   MRCLASS = {20F32 (20F10)},
  MRNUMBER = {1314099},
MRREVIEWER = {Ian M. Chiswell},
       DOI = {10.1006/jabr.1995.1010},
       URL = {https://doi.org/10.1006/jabr.1995.1010},
}
\end{document}